\documentclass[12pt]{amsart}

\usepackage{amssymb}

\usepackage{enumitem}

\usepackage{tikz}

\usepackage{amsmath}
\usepackage[english]{babel}

\usepackage{amsfonts,latexsym}
\usepackage{comment}
\usepackage{xfrac}
\usepackage[normalem]{ulem}

\makeatletter
\@namedef{subjclassname@2020}{%
  \textup{2020} Mathematics Subject Classification}
\makeatother

\usepackage[T1]{fontenc}

\newtheorem{theorem}{Theorem}[section]

\newtheorem{lemma}[theorem]{Lemma}
\newtheorem{proposition}[theorem]{Proposition}

\theoremstyle{definition}
\newtheorem{definition}[theorem]{Definition}

\numberwithin{equation}{section}

\def\cal{\mathcal}

\def\E{\mathbb{E}}
\def\Pr{\mathbb{P}}

\def \cal {\mathcal}

\def \A {\mathcal{A}_\infty(\mathcal{V})}
\def \und {\underline}
\def \Ed {\und{E}}

\newcounter{PropA}
\renewcommand{\thePropA}{\Alph{PropA}}
\newenvironment{classicalProp}[1]{\refstepcounter{PropA}   
 {\textbf{#1 {\thePropA}.   }}}{
  \ignorespacesafterend}
  \newcounter{ThA}
\renewcommand{\theThA}{\Alph{ThA}}
\newenvironment{classicalThm}[1]{\refstepcounter{ThA}   
 {\textbf{#1 {\theThA}.   }}}{
  \ignorespacesafterend}
\begin{document}


\frenchspacing

\textwidth=13.5cm
\textheight=23cm
\parindent=16pt
\oddsidemargin=-0.5cm
\evensidemargin=-0.5cm
\topmargin=-0.5cm


\title[Gaussian Behavior of Quadratic Irrationals]{Gaussian Behavior of Quadratic Irrationals}

\author[E. Cesaratto]{Eda Cesaratto}
\address{Instituto del Desarrollo Humano\\
Universidad Nacional de Gral. Sarmiento and National Council of Science and Technology (CONICET)\\ 
J.M. Guti\'errez 1150\\
(B1613GSX) Los Polvorines, Buenos Aires, Argentina}
\email{eda.cesaratto@campus.ungs.edu.ar}

\author[B. Vall\'ee]{Brigitte Vall\'ee}
\address{GREYC,  Universit\'e de Caen and CNRS\\
Sciences 3. Informatique, Universi\-t\'e de Caen, Bd Mar\'echal Juin\\ F-14032 Caen Cedex, France} 
\email{brigitte.vallee@unicaen.fr}

\date{}

\begin{abstract}
\baselineskip15.5pt

We study the probabilistic behavior of the continued fraction expansion of  a quadratic irrational number, when weighted  by  some
``additive'' cost. We prove  asymptotic Gaussian limit laws, with an optimal speed of convergence. We deal with the underlying dynamical system associated  with the Gauss map, and  its weighted periodic trajectories. We work with analytic combinatorics methods,  and mainly  bivariate Dirichlet generating functions;  we use various tools, from number theory (the Landau Theorem),  probability (the Quasi-Powers Theorem), or  dynamical  systems: our main object of study is the (weighted) transfer operator, that  we relate with the generating functions of interest.  The present  paper exhibits a strong  parallelism between periodic trajectories and rational trajectories. We indeed extend   the general framework which  has been previously described by  Baladi and Vall\'ee   for  rational trajectories.  However, our extension  to quadratic irrationals needs deeper  functional analysis properties. 
\end{abstract}


\maketitle


\baselineskip=17pt

\section{Introduction}

\subsection {General framework.}  \label{genfra}

This paper studies 
 particular quadratic irrationals (rqi in shorthand)  defined by purely periodic continued fraction expansions.  
 When a cost is defined on continued fraction expansions, notably  in an ``additive'' way,  each rqi  number  inherits  the cost defined on its period, and the paper focus on   such costs,   with a ``moderate growth''. Besides, we consider the usual notion of size that is  defined  in Number Theory contexts for a rqi number,  and  closely related to the fundamental unit of the associated quadratic field, and  to  the length of the corresponding closed  geodesic in the modular surface. 
 
  Our general framework is then  described as follows: we consider  the set ${\cal P}$ of   rqi numbers,   the subset ${\cal P}_N$  of rqi numbers with a size at most $N$, and the restriction $C_N$  to ${\cal P}_N$ of an additive  cost $C$ defined on ${\cal P}$. 
  We perform a probabilistic  analysis of the cost $C_N$ when the set ${\cal P}_N$ is endowed with the uniform probability. We prove in Theorem~\ref{Thm:mainCF} that the cost $C_N$ asymptotically follows a Gaussian law (when $N \to \infty$).  Moreover,  Theorem~\ref{Thm:mainCF} exhibits precise asymptotic expansions  for  its expectation and its variance.   Both are of order $\Theta(\log N)$. Theorem~\ref{Thm:mainconstants} provides  a computable description of the  constants that appear in the expansions. The main term of the expectation is already well studied, but,  to the best of our knowledge,  the estimates on the variance  are new, and this is the first study which exhibits   a limit  Gaussian law in such a context. 
  
  All these results are first proven within the unrestricted framework, when there do not exist ``constraints'' on  the continued fraction expansion of the rqi numbers.   But we also present  in Section~\ref{Sect:extensions}  similar results that hold for rqi numbers with bounded digits.  These results   combine the general methods of this paper with previous results that are specific to the constrained case,  described in   \cite{He3,Va98,CeVa}. 
  In the same section,  we also study an important (non additive) cost, namely the L\'evy constant, and obtain asymptotic expansions for its expectation and its variance.  Finally, we discuss the occurrence of local limit laws,  according to the type of costs (lattice or non-lattice).
 
  The central object in  our analysis is a bivariate  Dirichlet generating function $P(s,w)$ 
 where the complex parameter $s$ ``marks''  the size   and the complex parameter $w$ ``marks'' the cost.
  Our general methodology  is described along three main steps: 
  
{\em Step 1.} Relate the series $P(s, w)$ to   the weighted transfer operator ${\bf H}_{t,w}$,    associated with the underlying dynamical system defined by the Gauss map (with 
$t=s/2$).
 
{\em Step 2. } Transfer the analytic properties of the  operator  ${\bf H}_{t,w}$ (acting on convenient functional spaces) to obtain analytic properties of $P(s, w)$  about: \\
--    its  dominant  singularity (when $s$ is close to the real axis), \\
-- and its polynomial growth (when $s$ is far from the real axis).

{\em Step 3. } Extract asymptotics of coefficients   (with  the Landau Theorem) and exhibit  a quasi-powers behavior for the coefficients. With the Quasi-Powers Theorem, this entails the Gaussian law.
 
  There are strong similarities with the approach introduced by Baladi and Vall\'ee in \cite{BaVa}, 
  who perform  a  similar probabilistic analysis on  rationals, with an analogous series $Q(s, w)$. This  explains the strong similarity between the two results:    the Gaussian law obtained in the present paper, together with the precise asymptotic expansion of the expectation and the variance,  is of the same type as in the rational case.

Even though  Step 3  is  the same in the two analyses, 
 the present  rqi  framework   introduces  new  important issues in the first two steps.   In  Step 1, the periodicity phenomenon leads to several Dirichlet series -- not only the initial series $P(s, w)$--   which  take into account the 
 ``primitivity'' vs the ``non-primitivity''  of the period (as it is described in Section \ref{vargenfunct}).   In Step 2,  we  deal with two different functional spaces: 
 
-- 
when the parameter $s$ is close to the real axis,   we relate the bivariate series $P(s, w)$  to  {\em  traces} of  operators.  We then  need  spaces of analytic functions, where  such traces are   well-defined, as  it is shown in the works of  Mayer  \cite{Magreen}.   

-- when the parameter $s$ is far from the real axis,   the useful space is the space  ${\cal C}^1({\cal I})$ of functions on the unit interval, as in \cite{BaVa}.   Here, the rqi framework leads us to adapt  results  due to Pollicott and Sharp  to  obtain  polynomial growth (see  Theorem \ref{Thm:Far}).  

 Finally,   whereas the  Baladi-Vall\'ee analysis in \cite{BaVa}  uses a unique Dirichlet series $Q(s, w)$ and ``stays'' in the functional space ${\cal C}^1({\cal I})$,   the  more involved present  analysis deals with  various bivariate generating functions and two different functional spaces. Each generating function, and   each functional space plays a specific role.  The 
 occurrence of  the trace in the present study  explains the resemblances and the differences between constants which appear in the asymptotic estimates of the expectation and the variance, in the rqi case and in the rational case. 
  This is made precise in Section~\ref{Sect:constants}.

\subsection {Continued fractions and quadratic irrational numbers.} \label{S1.2}

 Every real number  
 $x \in ]0,1[$  admits  a continued fraction expansion of the form 
\[x=
\cfrac{1}{m_1+
\cfrac{1}{m_2+
\cfrac{1}{m_3+\dotsb
}}}
\]
denoted as  $x=[m_1,m_2, \dots]$. Here,  the coefficients $m_k$ are positive integers known as partial quotients and also called digits. Rational numbers have a finite 
continued fraction expansion. When the number $x$ is irrational, the continued fraction expansion is infinite and $x$ is  completely determined by the  whole sequence $(m_k)$ of its digits.

 In 1770, Lagrange proved  that a number is  quadratic irrational  if and only if its continued fraction expansion is eventually periodic.
We are mainly interested here in  particular quadratic irrational numbers,    whose continued fraction expansion is  purely  periodic.  Such numbers are  called 
\emph{reduced}    ({rqi}  in shorthand).  
This paper is devoted to 
their study: in particular, we describe the probabilistic behavior of such numbers via 
``costs'' that are defined   via their  continued fraction expansion.

{\bf \em Size.}  If $p(x)$ is the length of  the  smallest period  of a  rqi number $x$,  then  the number $x$  is  defined by the relation $[m_1, m_2, \ldots, m_k+x ] = x$, with $k = p(x)$ and
is denoted by $\langle m_1, m_2, \ldots, m_p \rangle$. This relation  rewrites as a quadratic polynomial  
equation of the form 
$A x^2+ Bx +C = 0$ with a triple 
$(A,  B, C)$  of relatively prime integers. Then $x$ belongs to  the quadratic field ${\mathbb Q}(\sqrt  \Delta)$ where $\Delta= B^2-4 AC$ is the discriminant of the polynomial, and one associates with $x$ the fundamental unit  $\epsilon (x)>1$ of this quadratic field. 
This fundamental unit  $\epsilon (x)$  plays, for  rqi numbers,   the same  role   as the denominator for rationals, and it defines a natural notion of size.

In this paper,  we consider the set  ${\cal P}$ of  rqi numbers  $x$, together with its  
finite\footnote{The finiteness  of ${\cal P}_N$ is  proven at the end of  Section 2.3.} subsets ${\cal P}_N$ formed with rqi numbers   of size  $\epsilon (x)$ at most $N$, 
defined as \begin{equation}\label{Eq:setP}
  \mathcal{P} := \{x\in \mathcal{I} \mid x \hbox{ is  a rqi number }\}, \qquad 
 \mathcal{P}_N :=\{x\in  \mathcal{P}\mid \ \epsilon(x)\le N\}  \, .
 \end{equation}

{\bf \em Cost on digits.} Any  non zero map $c: \mathbb N\rightarrow \mathbb{R}_{\ge 0}$ is  called a \emph{digit-cost}.  With a given digit-cost $c$,  we associate the \emph{cost}  $C(x)$ of a rqi  number $x$ defined  as the ``total cost'' on the (smallest) period of $x$, namely    
\begin{equation}
C(x):= c(m_1)+\dots +c(m_p),\quad \hbox{ if } x=\langle m_1,m_2,  \dots m_p \rangle \
\end{equation}
This defines a  non zero cost $C:  {\cal P} \rightarrow \mathbb{R}_{\ge 0}$. Moreover, 
we restrict ourselves to  digit-costs $c$   that satisfy   $c(m)=O(\log m)$ and  are said to be of {\em moderate growth}\footnote{The reasons for such a restriction are explained at the end of  Section \ref{Sect:33}. }.

There are  three main  examples of   digit-costs $c$ that can be treated by our methods, 
namely: the unit cost    $c =  1$;   the characteristic function  $\chi_n$ of  a  given digit $n$;  
the   length  $\ell$ of the  binary expansion of integers, defined as    $ \ell (m) :=   \lfloor \log_2 m \rfloor +1$.   The   associated costs $C$ are  natural and interesting. 
 For $c  = 1$, $C(x)$ coincides with the period-length $p(x)$.  For $c =  \chi_n$,  
 $C(x)$ equals  the number of digits equal to $n$ in the  smallest period of $x$. 
 Finally,   for $c  =   \ell$,  $C(x)$  is   the number of binary digits needed to store the rqi $x$.

 {\bf \em  Interpretation in the hyperbolic plane.}
 We  now describe the classical geometric  interpretation   
of  the size $\epsilon (x)$ and the digits $m_j(x)$ of a rqi number $x$ in  the hyperbolic plane.

\smallskip 
 We  first recall the coding of a curve in the hyperbolic plane ${\mathbb H}:=
 \{ z= x +iy \in {\mathbb C} \mid y >0 \}$.  
 When  ${\mathbb H}$  is endowed with  the usual metric $ds = |dz|/y$,  
 the geodesics are  the semi-circles centered on the real axis together with the vertical lines.
  Consider the hyperbolic triangle   $\Delta $ with the three cusps $i \infty$, $0$ and 1, represented in Fig.  \ref{fi1}. 
 Together with  all the triangles $ h(\Delta )$ which are the transforms of $\Delta$
 with $h \in SL_2({\mathbb Z})$, it  defines the   Farey tessellation  of ${\mathbb H}$ represented in Fig. \ref{fi2}.
 When a ``good''  oriented curve  $\gamma$  of ${\mathbb H}$ goes through such a triangle $ h(\Delta)$ in a general position, 
 there are two possibilities  for the position of   the curve  $h(\Delta) \cap \gamma$  with respect to the  three 
 cusps $h(i \infty)$, $h(0)$ and $h(1)$: if there is only one cusp on the right of the curve (and then two cusps on the left), 
  we code the portion of the curve  by $R$;  otherwise, we code it by $L$ (see Fig. \ref{fi1}). 

 \begin{figure}

  \begin{tikzpicture}
  \begin{scope}[scale=1.5]

 \draw[dashed,->] (-.5,0)--(1.5,0); 
  \draw[fill] (0,0) circle (.2pt);
  \node[below] at (0,0) {$0$};

   \coordinate (x) at (0,2);
\coordinate[] (y) at (1/2,2.2);
\coordinate (w) at (1,2);;

\coordinate[label=below:{$1$}] (z) at (1,0);

\fill[black!10] (0,0)--(x) .. controls +(up:.12cm) and +(left:.1cm) ..  (y) .. controls+(right:.1cm) and +(up:.12cm) .. (w)-- (1,0)
 arc (0:180:.5)-- cycle;
\node at (.5,1.3) {$\Delta$}; 

 \draw[->] (0,0)--(0,2); 

 \draw (1,0)--(1,2);;
    
  \draw (1,0) arc (0:180:.5);

\end{scope}
\begin{scope}[xshift=4cm]

 \begin{scope}[scale=3]
   \draw (1,0) arc (0:180:.5)-- (0,0)arc (180:0:.25)--(1/2,0)arc (180:0:.25)--cycle;
   \end{scope}
 \draw (1,0) arc (0:40:3) node [right,midway]  {$L$} ;
  \draw[-<] (1,0)arc (0:20:3) ;
 
 \end{scope}
 
\begin{scope}[xshift=7.5cm]

 \begin{scope}[scale=3]

   \draw (1,0) arc (0:180:.5)-- (0,0)arc (180:0:.25)--(1/2,0)arc (180:0:.25)--cycle;
   \end{scope}
  \draw (2,0)arc (0:40:3) node [right,midway]  {$R$} ;
    \draw[-<] (2,0)arc (0:20:3) ;

 \end{scope}
 
\end{tikzpicture}

\caption{On the left, the fundamental tile  $\Delta$ of the Farey tessellation. 
On the right,  the labeling of a  portion of a curve that intersects a triangle.}
\label{fi1}
\end{figure}

\smallskip
We  now return to a rqi $x$, and  consider its minimal even period,  of length $e(x)$,  with  $e(x) := p(x)$ if $p(x)$ is even, and length $e(x) := 2 p(x)$ if $p(x)$ is odd.  
First, if $x \in ]0, 1[ $ is a rqi number, its conjugate $\bar x$ satisfies $\bar x <-1$ and $$  x= \langle m_1, m_2, \ldots, m_e\rangle \quad  \Longrightarrow  \quad  -1/\bar x = \langle m_e, m_{e-1}, \ldots, m_2, m_1\rangle \, . $$

The geodesic $\gamma(x)$  which links   $\bar x$ to $ x$ (with this orientation)  
intersects the imaginary axis at $i t(x)$. This defines   two oriented curves, the curve $\gamma_+(x)$ that links $i t(x)$ to $x$, and the curve $\gamma_-(x)$ that links $it(x)$ to $\bar x$. As Series \cite{Se85}  and Pollicott  \cite{Po86} show,  the coding of the  geodesics 
``copies'' the continued expansions of the  associated rqi numbers,  and 
the codings of the  two curves  $\gamma_+(x)$ and  $\gamma_+(-1/ \bar x)$ are  the  respective periodic infinite words 
$$\left( L^{m_1} R^{m_2}L^{m_3}\ldots R^{m_e}\right)^{\mathbb N},   \qquad \left(L^{m_e}  R^{m_{e-1}} L^{m_{e-2}}\ldots R^{m_1} \right)^{\mathbb N}\, .$$   The  close connection  that relates  the codings of  the two curves $\gamma_+(-1/ \bar x)$ and $- \gamma_-(x)$  finally entails   a coding for the concatenation $\gamma(x)$ of the two curves $(- \gamma_-(x)) $ and $\gamma_+(x)$ as the bi-infinite  periodic word  of period   $R^{m_1} L^{m_2}R^{m_3}\ldots L^{m_e}$
(see Fig. \ref{fi2}).\\
 Moreover,  the length of  the ``primitive'' part of the geodesic  $\gamma(x)$ -- associated with
 the coding  $L^{m_1} R^{m_2}L^{m_3}\ldots R^{m_e}$ of the minimal even period -- equals
 2 $\log \epsilon (x)$, where $\epsilon (x)$ is  the fundamental unit $\epsilon (x)$  associated with the rqi $x$. 
 
 \smallskip 
  These interpretations of the size $\epsilon (x)$ and the digits $m_j(x)$ of a rqi $x$  on its geodesic $\gamma(x)$ 
  provide  a geometric  framework  for  the whole present study. 
  
\begin{figure}[h]

\begin{tikzpicture}[scale=1.8]
 \draw[dashed] (1,0)--(4,0); 
  \node[below] at (3,0) {$0$};
 \draw[-] (3,0)--(3,2); 
 \draw[-] (4,0)--(4,2);
 \draw[-] (2,0)--(2,2);
 \draw[-] (1,0)--(1,2);
  
   \coordinate (x) at (0,2);
   \coordinate (x1) at (1,2);
\coordinate[] (y) at (1/2,2.2);
\coordinate[] (y1) at (3/2,2.2);
\coordinate (w) at (1,2);
\coordinate (w1) at (2,2);

\coordinate[label=right:$it(x)$] (z) at (3.2,1.1);
\path [->] (z) edge (3.01,1);

\begin{scope}[xshift=3cm]
\coordinate (a) at ({1/(1+1/(1+1/(1+1/2)))},0);
\node[below] at (a) {$x$};

\draw[thick] (a) arc (0:180:{1/2+1/(1+1/(1+1/2))}) node[below]{$\overline{x}$};
\end{scope}

\node at (3.5,1.5) {$\Delta$};

\begin{scope}[xshift=2cm]
\draw (.5,0) arc (0:180:1/12);
\draw (1/3,0) arc (0:180:1/6);
\draw (1/3,0) arc (0:180:1/24);
\draw (1/4,0) arc (0:180:1/8);
\draw (1/2,0) arc (0:180:1/20);
\draw (2/5,0) arc (0:180:1/30);
\draw (1,0) arc (0:180:.5)-- (0,0)arc (180:0:.25)--(1/2,0)arc (180:0:.25);

 \draw (1,0) arc (0:180:1/6);
  \draw (2/3,0) arc (0:180:1/12);
\draw (1,0) arc (0:180:1/8);
\draw (3/4,0) arc (0:180:1/24);
\draw (2/3,0) arc (0:180:1/30);
\draw (3/5,0) arc (0:180:1/20);

\end{scope}

\begin{scope}[xshift=3cm]
\draw (.5,0) arc (0:180:1/12);
\draw (1/3,0) arc (0:180:1/6);
\draw (1/3,0) arc (0:180:1/24);
\draw (1/4,0) arc (0:180:1/8);
\draw (1/2,0) arc (0:180:1/20);
\draw (2/5,0) arc (0:180:1/30);
\draw (1,0) arc (0:180:.5)-- (0,0)arc (180:0:.25)--(1/2,0)arc (180:0:.25);

 \draw (1,0) arc (0:180:1/6);
  \draw (2/3,0) arc (0:180:1/12);
\draw (1,0) arc (0:180:1/8);
\draw (3/4,0) arc (0:180:1/24);
\draw (2/3,0) arc (0:180:1/30);
\draw (3/5,0) arc (0:180:1/20);

\end{scope}

\begin{scope}[xshift=1cm]
\draw (.5,0) arc (0:180:1/12);
\draw (1/3,0) arc (0:180:1/6);
\draw (1/3,0) arc (0:180:1/24);
\draw (1/4,0) arc (0:180:1/8);
\draw (1/2,0) arc (0:180:1/20);
\draw (2/5,0) arc (0:180:1/30);
\draw (1,0) arc (0:180:.5)-- (0,0)arc (180:0:.25)--(1/2,0)arc (180:0:.25);

 \draw (1,0) arc (0:180:1/6);
  \draw (2/3,0) arc (0:180:1/12);
\draw (1,0) arc (0:180:1/8);
\draw (3/4,0) arc (0:180:1/24);
\draw (2/3,0) arc (0:180:1/30);
\draw (3/5,0) arc (0:180:1/20);
 
\end{scope}  

\begin{scope}[xshift=5cm,scale=2]
\coordinate (a) at ({1/(1+1/(1+1/(1+1/2)))},0);
\node[below] at (a) {$x$};
\draw (a) arc (0:69:{1/2+1/(1+1/(1+1/2))});
\draw (a) arc (0:15:{1/2+1/(1+1/(1+1/2))}) node [right,midway] {$L$};
\draw[-<] (a) arc (0:15:{1/2+1/(1+1/(1+1/2))})arc(15:25:{1/2+1/(1+1/(1+1/2))})node [right,midway] {$R$};
\draw (a) arc (0:15:{1/2+1/(1+1/(1+1/2))})arc(15:30:{1/2+1/(1+1/(1+1/2))})arc(30:60:{1/2+1/(1+1/(1+1/2))})node [right,midway] {$L$};

\coordinate[label=right:$it(x)$] (z) at (0.2,1.1);
\path [->] (z) edge (0.01,1);

 \draw[dashed] (0,0)--(1,0); 
  \draw[fill] (0,0) circle (.2pt);
  \node[below] at (0,0) {$0$};
  \node[below] at (1,0) {$1$};
 \draw[-] (0,0)--(0,1.2); 
 \draw[-] (1,0)--(1,1.2);
\draw (.5,0) arc (0:180:1/12);
\draw (1/3,0) arc (0:180:1/6);
\draw (1/3,0) arc (0:180:1/24);
\draw (1/4,0) arc (0:180:1/8);
\draw (1/2,0) arc (0:180:1/20);
\draw (2/5,0) arc (0:180:1/30);
\draw (1,0) arc (0:180:.5)-- (0,0)arc (180:0:.25)--(1/2,0)arc (180:0:.25);

 \draw (1,0) arc (0:180:1/6);
  \draw (2/3,0) arc (0:180:1/12);
\draw (1,0) arc (0:180:1/8);
\draw (3/4,0) arc (0:180:1/24);
\draw (2/3,0) arc (0:180:1/30);
\draw (3/5,0) arc (0:180:1/20);

\end{scope}

\end{tikzpicture}

\caption{Farey tessellation and continued fractions.  Here,  we consider
$x:= \phi^{-1}$ (here $\phi$ is the golden ratio)  and the geodesic which links $x$ to its conjugate $- \phi$ as it goes through  the Farey tessellation. On the right, the coding of $x$ in terms of $R$ and $L$.    }
\label{fi2}
\end{figure}


\subsection{Statement of the main results.}
We study the  probabilistic behavior of the  continued fraction expansion of  a rqi number, with respect to some cost $C$.  When  the set
$\mathcal{P}_N$ is endowed with the uniform probability $\mathbb P_N$, the restriction   $C_N$    of  the cost $C$ to  $\mathcal{P}_N$ becomes a random variable.  Our  first main result  exhibits the asymptotic Gaussian behavior for the sequence $(C_N)$.  

 \begin{theorem}\label{Thm:mainCF}
Consider the set ${\cal P}$ of the reduced quadratic irrational numbers $x$, endowed with the size $\epsilon$. With  a 
non zero cost  of moderate growth,  associate the additive cost $C$  on the set ${\cal P}$. 
Then, the following holds on the set ${\cal P}_N$ of  reduced quadratic irrational numbers $x$ with $\epsilon(x) \le N$,  for $N \to \infty$:  

\begin{itemize} 

  \item[$(i)$] 
There  are  four constants, two dominant constants   $\mu(c)$ and  $\nu(c)$ that are strictly positive, and two subdominant constants, $\mu_1(c)$ and $\nu_1(c)$,  for which the  expectation $\E_N[C]$ and the variance $ {\mathbb V}_N[C]$  satisfy the following  asymptotic estimates,    for some  $\beta >0$, 
\begin{equation*}
 \E_N[C]= \mu(c) \log N + \mu_1(c) +O(N^{-\beta}), \ 
 \end{equation*}
\begin{equation*}
 {\mathbb V}_N[C]=\nu(c) \log N + \nu_1(c) +O(N^{-\beta}). 
\end{equation*}

\smallskip 
\item[$(ii)$] Moreover,  the  distribution of  $C$ 
is asymptotically Gaussian,  
$$\Pr _N\left[ x  \mid \frac {C(x) - \mu(c) \log N} { \sqrt {\nu(c)\log N} } \le t \right]= \frac 1 { \sqrt {2\pi}} \int_{-\infty} ^t e^{-u^2 /2} du + O\left( \frac 1 {\sqrt {\log N}} \right)\,.$$

 \end{itemize}
 \end{theorem}  
 Remark that the speed of convergence  towards the Gaussian law  is optimal,  of the same order   $(\log N)^{-1/2}$ as in the usual Central Limit Theorem.

\subsection{Main constants of the analysis.}  The next theorem 
gives a precise description of the four constants  that appear in Theorem \ref{Thm:mainCF}.  We gather here, 
various results of different types. The first two items $(i)$ and $(ii)$ are classical; items $(iii)$ and $(iv)$ show a strong analogy with results in \cite{BaVa}.  
 
 \begin{theorem}\label{Thm:mainconstants}   Consider the  digit-cost $c$ of moderate growth associated with the additive cost $C$.
 The following holds:  
 \begin{itemize} 
 \item [$(i)$]  Consider the  dynamical system  related to  the Gauss map that underlies the continued fraction expansion, and    the (weighted) transfer operator ${\bf H}_{t, w}$
  which  associates with   a function $f$ the function  ${\bf H}_{t, w} [f]$,  defined as 
 \begin{equation} \label {debH} {\bf H}_{t, w} [f] (x)  := \sum_{m \ge 1} \frac { e^{ w\, c(m)}   }{(m+x)^{2t} } f\left( \frac 1 {m+x}\right)\, . 
\end{equation}
 When it acts on the functional space $\A$ defined in \eqref{calA},    the operator ${\bf H}_{t, w}$ possesses,  for $(t, w)$ close to $(1, 0)$, a dominant eigenvalue
denoted as $\lambda(t, w)$. 

\smallskip
\item[$(ii)$]  
The   two partial derivatives of  the map $(t, w) \mapsto \lambda(t, w)$  at  $(t, w) = (1, 0)$ are related to  the entropy ${\cal E}$ of the system together  with the mean value $\E[c]$ of the cost with respect to the Gauss density $\psi : x \mapsto (1/\log 2) (1/(1+x))$, 
$$ -\lambda_t'(1,0)  = {\cal E} = \frac {\pi^2} {6 \log 2}, \quad  \lambda_w'(1,0)  =   \E[c]  = \sum_{m \ge 1} c(m) \int_{1/m}^{1/(m+1)}\psi(x) dx   \, .  $$

\smallskip
\item[$(iii)$]  Consider the map $w \mapsto \sigma(w)$ that is  defined from the  equation $\lambda(\sigma(w), w) = 1$ and satisfies 
$\sigma(0)  = 1$, together with   the two mappings $$U(w)  :=  2(\sigma(w) -1),  \qquad V(w)  :=  \log \left(  \frac {\lambda_t'(1,0)}{\lambda_t'(\sigma(w),w)}\right) - \log \sigma(w)\, .$$
 The  constants  of Theorem \ref{Thm:mainCF} are expressed  with the first two derivatives of $U$ and $V$, namely, 
$$\mu(c) = U'(0) =  \frac {2}{\cal E} \,  { \E[c]} , 
\quad \nu(c) = U''(0), \quad \mu_1(c) = V'(0), \quad \nu_1(c) = V''(0).$$

\smallskip
\item[$(iv)$] The  dominant constants $\mu(c), \nu(c)$  coincide with their analogs that occur in the analysis of rational trajectories described in \cite{BaVa}.

\end{itemize}
\end{theorem}

  We first  observe  
   the equality  $  \mu(c) =  \left( {2}/{\cal E}\right)  { \E[c]}$   that  entails the estimate  $\E_N[C] \sim \E_N[p] \cdot \E[c]$. This estimate  may be compared to the ergodic relation  that holds on almost any generic  trajectories. This  leads to the following  (informal) statement ``The periodic trajectories behave on average as the generic trajectories behave almost everywhere'', which holds in a very general  context of ``dynamical analysis''.

   All the constants that appear in Theorem \ref{Thm:mainCF}  are computable. In particular,  the constants $\E[c]$  related to the three  costs of interest are
$$  \E [c] = 1, \ \ \E [\chi_n] =  \frac 1 {\log 2} \log \left[ \frac {(n+1) ^2} {n(n+2)} \right], \ \  \E[\ell] = \frac 1 {\log 2}   \prod _{i = 1} ^\infty \log \left( 1 + \frac 1 {2^k}\right).$$
 In many natural situations, the mean value $\E [c]$  admits an  explicit expression, and this is thus the same for the  constant $\mu(c)$ itself.  
 
 The situation is completely different for the other three constants, as there are no longer close expressions for the  second (and the third)  derivatives of the mapping $(t, w) \mapsto \lambda(t, w)$ at $(1, 0)$.  However, these derivatives are computable in polynomial time with  methods  that  were first described  in  \cite{He3} and  \cite{DFV},  then further developed in \cite{Lh1}. All these methods are based on a very natural idea: use the  truncated Taylor expansion of analytic functions to approximate the operator  by matrices (that  now act on polynomials);  compute the spectrum of the matrices with  classical tools of linear algebra (in finite dimensions). This   should   provide approximates of the   (upper part of) the spectrum of the operator (and this is proven to be true).

\subsection {Description of the methods. }  The present work deals with methods that come from analytic combinatorics, in the same lines as  in the book of Flajolet and Sedgewick \cite{FlSe}. We associate with  the  set ${\cal P} $,  endowed with the size $\epsilon$ and the cost $C$, a  bivariate generating function;  in the present number theory framework, it is of Dirichlet type, 
\begin{equation*}\label{Eq:defP}P(s,w) := \sum_{x\in \mathcal P}\exp(w C(x))\, \epsilon(x)^{-s}\ .
 \end{equation*} 
 Cumulative costs over the set 
${\cal P}_N$ are  written as sums of coefficients,   
and we notably define  
 \begin{equation*}
 S^{[C]}_w(N) := \sum_{\substack{x \in {\cal P} \\{ \epsilon(x)\le N} }}\exp(w C(x)) \, ;  \end{equation*}
the moment generating of the cost  $C$ on ${\cal P}_N$ is  then written as  the quotient of such sums,  
 \begin{equation*} \mathbb{E}_N[\exp(wC)]   =  \frac {S^{[C]}_w(N)}{S^{[C]}_0(N)} \, .
 \end{equation*}
 The distribution of  the cost $C$ on ${\cal P}_N$  is indeed studied here with this moment generating function: we use the Quasi-Powers Theorem \cite{Hw96}   that states that  an asymptotic Gaussian law for  $C$  on ${\cal P}_N$ holds as soon as the moment generating function has a ``uniform quasi-powers'' form (when $w$ is a complex  number near 0). \\
 We thus  need  precise  asymptotic estimates on the sums $S^{[C]}_w(N)$ of the series $P(s, w)$; analytic combinatorics principles relate them  to the analytic properties of the series $P(s, w)$. The strong tool which operates this transfer  here, between  analytic properties of $P(s, w)$ and asymptotic properties of its coefficients,  is the Landau Theorem proven by Landau in  \cite{Lan}. The analytical properties  that are needed  for $s\mapsto P(s, w)$  are of two types: a precise knowledge of its  singularities (here, its poles, located near the real axis)  together with  a good knowledge of its  behavior for $|\Im s| \to \infty$ (here, polynomial growth  for $|\Im s| \to \infty$). 
 
 We have then  to study the series $P(s, w)$. This will be done  with   ``dynamical  analysis'', which is already used in many works, in particular in the 
 analog rational study (see for instance   in \cite{BaVa,Va8,Va00}). 
 The idea  is  to use the underlying dynamical system, associated with the continued fraction transformation (the Gauss map), and relate the series $P(s, w)$ to the (weighted) transfer operator of the dynamical system described in \eqref{debH}.  
It is then needed to study the transfer operator itself,  for $w$  close to 0, and various values of $s$.

 Here, as already mentioned in Section \ref{genfra}, we   need  two different functional spaces: the first one, useful when $s$ is close to the real axis, is the  set  $\A$ of analytic functions (defined in \eqref{calA}) where the transfer operator admits dominant spectral properties and  a  well-defined trace (in the sense of Grothendieck); this study provides  a precise knowledge of the singularities  of  $s \mapsto P(s, w)$  located near the real axis;  the second functional space,  useful for  $|\Im s| \to \infty$, is the set of ${\cal C}^1$ functions,  where the transfer operator admits bounds \`a  la Dolgopyat \cite{Do}, which entail the  polynomial growth of $P(s, w)$   for  $|\Im s| \to \infty$. 
 
 \subsection {Comparison with the Baladi-Vall\'ee approach.} 
As already mentioned in Section \ref{genfra},   our   methods  are strongly similar  to those that  are used for rational trajectories  in \cite{BaVa}. The paper \cite{BaVa}  deals with  the generating function  of the rational set ${\cal Q}:= {\mathbb Q} \cap [0, 1[$, endowed with the denominator size $q(x)$, namely 
\begin{equation} \label {Qsw}
 Q(s, w) :=  \sum_{x\in \mathcal Q}\exp(w C(x))\, q(x)^{-s} \, . 
 \end{equation}
 The series is  directly  expressed with the weighted transfer operator ${\bf H}_{t, w}$ introduced in \eqref{debH}, as the equality $Q(s, w) = (I-{\bf H}_{t, w})^{-1}[1](0)$ holds\footnote{This is not completely exact, as the final operator ${\bf F}_{t, w}$ also intervenes (see Section \ref{Sect:constants}).}, with $t= s/2$.   This explains why  the  first two steps of the rational study   are easier than  in the present one,   that is  more involved due to the three issues described in  Section \ref{genfra}. 
 
  Along Step 3, the main tools are  essentially the same, namely,  the Landau Theorem\footnote{The paper  \cite{BaVa} does not use  the  ``ready for use''  version of  the Landau Theorem   described in Section \ref{Sec:Landau}, and, as  in many  other works (see for instance \cite{PoSh}), it proves the analog result ``by hands''.},  and   the Quasi-Powers Theorem.

\subsection{Comparison with already known results.}\label{Sect:comparison}
 \    \   \vskip 0.1cm 
{\bf \em Plain periodic trajectories.} The periodic trajectories of dynamical systems are very well studied, by means of various zeta series, as Ruelle introduced it   in his  pioneering work  \cite{Ru76}. In the case of the Gauss map, the Selberg zeta series  is a powerful tool  that is well-adapted to plain periodic trajectories (without a cost $C$), and there is a close connection exhibited by Mayer in \cite{Ma} between 
the plain Dirichlet series $P(s) := P(s, 0)$,  the Selberg zeta series and the zeta Riemann function.  This leads to  the asymptotic estimate of the cardinality $|{\cal P}_N| \sim  [(3\log 2)/ \pi^2 ] \cdot N^2$.  
The study of the remainder term is more involved and was performed by 
 Boca      in \cite{Bo07} 
who obtains  an error term of order $O_{\varepsilon}(N^{7/4+\varepsilon})$ for any $\varepsilon >0$.  His proof relies on  precise estimates  on  the  distribution of pairs  $(x, y)$ of coprime integers that satisfy $xy = 1 \mod q$  (for an integer $q$),  together with classical results on the Riemann zeta function. Following this approach, Ustinov \cite{Us13} improves the estimate of the error term and  obtains
an order $O(N^{3/2}\log^4 N )$.

Our series $P(s, w)$ can be viewed  as a ``twisted zeta Selberg series'', where the twist is brought by the cost $C$;   many geometric properties  of the Selberg zeta series  ``disappear'' with  this twist and  it is not clear  how to extend  Boca's and Ustinov's methods to  weighted  periodic trajectories.

{\bf \em Weighted periodic trajectories.} To the best of our knowledge, Parry and Pollicott in \cite{PaPo}, then  Pollicott  in \cite{Po86}  were the first to study  weighted periodic trajectories, with analytic combinatorics methods. Pollicott  indeed introduces a weighted generating function, and, as he only performs  average-case analysis,   he may  use Tauberian Theorems that provide estimates of the mean values, as in Theorem \ref{Thm:mainCF}$(i)$,  but with only the first (dominant) term,  without remainder terms. Later on, Faivre  in \cite {Fa} used slightly different methods, and obtained the same results.

With similar methods, Kelmer \cite{Ke12} studies the set $\mathcal{P}_N^{(n)}$  of rqi numbers  whose 
alternate sum of partial quotients is fixed and equal to $n$.   He  considers the  restriction  of  a    function $f \in \mathcal{C}([0,1])$ to the set   $\mathcal{P}_N^{(n)}$ and proves that, for $n$ fixed and $N \to \infty$,  its mean value  tends  to the  mean value  $\mathbb E [f]$ of $f$ with respect to the Gauss density $\psi$. He does not  make precise the speed of convergence.

{\bf \em Remainder terms?}  
One may expect more precise  asymptotic  estimates, for  the mean value (estimates with remainder terms),  or for the  variance, (with even a rough  estimate); such estimates are also   needed  for  a distributional study. They appear to be based on a more precise knowledge of the weighted transfer operator ${\bf H}_{t, w}$ for $|\Im t| \to \infty$. This better knowledge was indeed brought by works of Dolgopyat, notably in \cite{Do}, first for  {\em unweighted} transfer operators associated with  Markovian dynamical systems with a {\em finite} number of branches.  Then,  Pollicott and Sharp  have used Dolgopyat's result  (see for instance \cite{PoSh}), and  performed various average-case analyses, where they  obtain estimates for various mean values  (related to  restricted periodic trajectories) with remainder terms.

Dolgopyat's results  were further adapted  in \cite {BaVa} to  the present case of interest -- a weighted transfer operator    associated with  a Markovian dynamical system with a {\em  infinite denumerable} number of branches--.  Then,  in \cite{BaVa}, Baladi and Vall\'ee   have used  this extension of Dolgopyat's result  for  distributional studies of  weighted rational trajectories and derived asymptotic Gaussian laws. To the best of our knowledge, the present work is the first which is devoted to distributional studies of  weighted periodic trajectories. 

 \subsection {Plan of the paper.}    Section 2 
 presents the main objects, and the general methodology; it   introduces  the  various generating functions of interest, describes their  relations and their first properties.   Section 3  introduces  the weighted transfer operator,  the two  useful functional spaces $\A$ and  ${\cal C}^1({\cal I})$, and  describes their role in the analysis.     It  relates   analytic properties of the generating functions  to fine properties of the  weighted transfer operator:  bounds  \`a la Dolgopyat in the space ${\cal C}^1({\cal I})$, and  properties of the traces  in the space $\A$. As our work needs a precise (and long)  study of traces, we perform it in the Annex and just summarize the main results in Section 3.  
 Section 4  transfers  analytic properties of the bivariate generating functions into probabilistic properties of the cost: it  uses two main tools (the Landau Theorem, then  the Quasi-Powers Theorem) to  transfer analytic properties of the bivariate generating functions into probabilistic properties of the cost, notably 
 the asymptotic  Gaussian law (Theorem \ref{Thm:mainCF}). It also  performs a precise  study of the main constants of interest (Theorem \ref{Thm:mainconstants}). In Section 5, we explain how to ``transfer'' to periodic trajectories some results that  are already proven on rational trajectories.  We study in particular the case of ``constrained''  periodic trajectories.  The paper ends with an Annex, which is devoted to a self-contained study of traces of operators.   It is needed  in Section 3 and  may be of independent interest for people  that are not specialists  of this theory.

 {\em The  results that are already known and   directly used in the paper are labeled with letters, whereas the new results that are proven in the paper are labeled with numbers. }

\section{Presentation of the  main objects} 

The present section  introduces the main actors: first, in Section \ref{Sect:EDS}, the underlying  dynamical system associated with the Gauss map,  together with the main properties of its branches (Section \ref{rhoUNI}).  We then   recall  well-known facts about irrational quadratic numbers and fundamental units  in Section \ref{Sect:QIN} and describe the costs of interest  (of moderate growth)  (Section \ref{Sect:CMG32}). 
  {
  Then,  Section   \ref{vargenfunct}
   introduces  the  various (bivariate) generating functions of interest.  The initial generating function  $P(s, w)$ involves quadratic numbers, and  thus deals with a primitivity condition that is difficult to manage. We first  ``suppress'' this  primitivity condition,  then we change of size: we replace the initial size $\epsilon$  by  another (closely related) size $\alpha$, which  satisfies multiplicative properties. We   finally  obtain   a series $E(s, w)$. 
 The end of the  Section  describes the first easy properties of these generating functions and  explains  why it suffices to deal with $E(s,w)$ throughout the rest.  } 
 
\subsection{The Euclidean dynamical system}\label{Sect:EDS}
The continued fraction expansion    encodes the 
trajectories of the dynamical system associated with the Gauss map
$T: \mathcal{I} \rightarrow \mathcal{I}$,
\begin{equation}  {\cal I} := [0, 1], \qquad  T(x):=    \frac 1 x  -
    \left\lfloor \frac 1 x \right\rfloor, \quad \ \hbox { for } x \not = 0, \ \
    T(0)= 0 \, .
\label{standard}
\end{equation}
Here, $\lfloor x \rfloor$ is the integer part of
$x$. The   trajectory $${\cal T}(x) = (
T(x), T^2(x), \ldots, T^k(x),\ldots)$$ of  an irrational $x$ never meets $0$ and is
encoded by  the infinite sequence 
$$(m_{1}(x),
m_{2}(x), \ldots, m_{k}(x), \ldots) \  \hbox{with} \  m_{k}(x) :=m(T^{k-1}(x)),  \ \
m(x) :=\left\lfloor {\frac {1}{x}}\right\rfloor.$$
The map $T$  is a piecewise complete interval map,   with its set   $\mathcal{H}$ of inverse branches, 
$$
{\cal H} =   \left\{h_{m}\mid  x \mapsto \frac{1}{m+x}; \quad   m\ge 1 \right\}\ .
$$

For any $k \ge 1$, the set $\mathcal{H}^k$ of inverse branches of $T^k$  gathers LFTs  (linear fractional transformations) of the form 
\begin{equation}\label{Eq:hom}
 h=h_{m_1}\circ h_{m_2}\circ \dots \circ h_{m_k}\, .
\end{equation}
  By definition, the depth of  a LFT $h \in {\cal H}^k$ (denoted by $|h|$) is  equal to $k$.   The set 
$$\mathcal{H}^+  := \bigcup_{k\ge 1}\mathcal{H}^k\, $$
gathers the inverse branches of any  strictly positive depth.  For $h \in {\cal H}^+$,  the  interval  $h(\mathcal{I})$ is called  the fundamental interval relative to $h$. For  a LFT  $h$   described  in  \eqref{Eq:hom},  the interval $h({\cal I})$  gathers  reals $x$ that  have their $i$-th partial quotients equal to $m_i$ for any  $ i \in [1..k]$.  For any $k \ge 1$,  the  fundamental intervals $h({\cal I})$ for $h \in {\cal H}^k$  form a topological partition of the unit interval $\mathcal I$.

\subsection{First properties of the dynamical system.}\label{rhoUNI}   
Here,  we  describe  the main  properties (of geometric flavor) satisfied by the branches of the set ${\cal H}^+$ defined in Section \ref{Sect:EDS} and their derivatives. 
The intervals of interest are $\mathcal{I}:=[0,1]$ and $\mathcal{J}:=[-1/4, 9/4]$, the last one intervenes because of Property $(P3)$ associated with $r = 5/4$ .

\begin{enumerate}

\smallskip
\item[$(P1)$]   There is a constant $L$ (usually called the distortion constant -- here equal to $\sup (8/3, \exp (20/3))$--  
  for which the  two inequalities  hold  for   any   inverse branch  $h  \in {\cal H}^+$,  
$$|h''(x)|\le  L \, |h'(x)|, \quad \forall x \in {\cal J}, \qquad  \frac 1 L\le \frac{|h'(x)|}{ |h'(y)|} \le  L, \quad  \forall (x, y) \in {\cal J}^2\, .$$

\smallskip 
\item[ $(P2)$]   The map $T$ is  piecewise strongly expanding.  There is a contraction ratio 
$\rho$ equal to    $\phi^{-2}$  for which the   following inequalities  hold and involve the fixed point $x_h$ of  the inverse branch $h$ and the distortion constant $L$, namely 
$$|h'(x_h)|\le {\rho}^{|h|} , \,\qquad |h'(x)| \le L {\rho}^{|h|}  \quad   \forall x \in {\cal J}, \qquad \forall h \in {\cal H}^+ \, . $$

\smallskip
\item[ $(P3)$]    Consider the disk ${\cal V}_r$ of center 1 and radius $r$.  For  
$r \in ]1, (1+ \sqrt 5)/2[$, there exists $\tilde r < r$ for which  the following holds,    for any  branch $h \in {\cal H}^+$: 

\begin{itemize} 
\item [$(i)$]
It  is analytic on ${\cal V}_r$   and  the inclusion $ h({\cal V}_r) \subset {\cal V}_{ \tilde r} $  holds. 

\item[$(ii)$]The analytic extension $\und h$ of the map $x \mapsto |h'(x)|$ is non zero on ${\cal V}_r$ and  the function $z\mapsto  \log \und h(z)$ is analytic  on ${\cal V}_r$. 
\end{itemize}
We consider  in the following   the case $r = 5/4$ and denote ${\cal V}:= {\cal V}_{5/4}$.

\smallskip
\item[$(P4)$] The  two series 
$$ \sum_{h \in \mathcal{H}} \und h^s(z), \qquad \sum_{h\in \mathcal{H}} \log \und h(z) \, \und h^s(z),  $$ are absolutely convergent  for any $z$ in the closure of the disc $\overline{\cal V}$ and
 on the compact subsets of the half-plane $\Re s >1/2$.

\smallskip
\item[$(P5)$]   The UNI (Uniform Non Integrability) condition  holds. It is described  with
 the ``distance''  
$\Delta (h, k)$   between 
 two inverse branches $h$ and $k$  of same depth,  (introduced  in \cite {BaVa}), 
\begin{equation*} \label {delta}
\Delta (h, k)  :=  \inf_{x \in {\cal I}} |\Psi'_{h, k} (x)|,  \quad \hbox{ with } \quad 
  \Psi_{h, k} (x) := \log\frac {|h'(x)|}{|k'(x)|} \, ,
\end{equation*} 
 and the  measure $J(h, \eta)$ of the ``ball'' of center $h$ and radius $\eta>0$,  
\begin {equation*} \label {J}
J(h, \eta) :=  \sum_ {\substack{ k\in\cal H^{n} ,\\ {\Delta(h, k) \le  \eta}}}
|k({\cal I})| \, , \quad  \hbox { (for  $h \in {\cal H}^{n}$)} \, .
\end{equation*}
The condition  UNI  states the existence of  a constant $K$ for which 

\centerline{$
J(h, \rho ^{an})  \le K   \rho^{an}\, , \quad \forall a   \, (0<a  <  1), \ \ 
\forall n, \ \  \forall h \in {\cal H}^n \, . $}

\noindent There is also a technical statement  about the second derivatives $\Psi''_{h, k}(x)$ 

\centerline { 
$
Q:= \sup \{ |\Psi''_{h, k}(x) | \mid n \ge 1 \, , h,k \in \mathcal{H}^n, x  \in \cal I\} < 
\infty 
$. }

\end{enumerate}

\subsection{\bf Quadratic irrational numbers.}\label{Sect:QIN} 
Quadratic irrational numbers are  the only  numbers $x$ for which the trajectory  ${\cal T}(x)$ is  eventually periodic, and a {\em reduced } quadratic irrational number (rqi in shorthand) is associated with  a purely periodic trajectory. A rqi $x$    is then completely defined by its (smallest) period  $p(x)$  which defines a {\em primitive} cycle (of  the same length).

The set of rqi numbers coincides with the set of fixed points $x_h$  of LFTs $h\in\mathcal{H}^+$.  All  the LFTs which determine a given rqi number  $x$  are the powers of  a given LFT which is primitive. This defines a bijection between the  set ${\cal P}$ of rqi numbers   and the  set of primitive LFTs,  together with  a decomposition of ${\cal H}^+$ as a  disjoint sum, 
\begin{equation} \label {HP} {\mathcal H}^+=\bigcup_{k\ge 1 }{\mathcal P}^{(k)},\, \quad 
\mathcal{P}^{(k)}:=\{h^k,\, h\in \mathcal{P}\}\, .
\end{equation}
A fundamental  quantity associated with  a {rqi} number $x$ 
is the product of all the shifted $T^i(x)$ along its period. If $h$ is the primitive LFT which defines $x = x_h$, one  considers
$$ \alpha (x):= \prod_{i= 0}^{p(x)-1} \, T^i (x)  =   |h'(x_h)|^{1/2}\, .$$
Here,   the fundamental unit $\epsilon (x)$   is chosen as the size for  the rqi $x$ (see Section \ref{S1.2}). It  is related to  the  product  $\tilde \alpha (x)$ of  all the shifted $T^i(x)$ along  the {\em minimal even} period of length $e(x)$ (already   mentioned in Section \ref{S1.2}), and  the following equality  holds
$$\epsilon (x)= \tilde \alpha (x) ^{-1}, \qquad  \tilde \alpha (x):= \prod_{i= 0}^{e(x)-1} \, T^i (x) \, , $$
that also translates as 
$$ \epsilon (x) = \alpha (x)^{-r(x)}  \hbox{
  with  $r(x) :=  1$ for even $p(x)$,  and  $r(x) :=  2$ for odd $p(x)$}. $$  
  
  It is then natural (and useful) to extend the definition of $\alpha, \epsilon, r$ from the set  ${\cal P}$ of rqi numbers to the set ${\cal H}^+$. For any $h \in {\cal H}^+$ (primitive or not),  we define $\alpha(h)$  as 
\begin{equation} 
\label{alpha}\alpha(h) :=   \prod_{i= 0}^{|h|-1} \, T^i (x_h)  = |h'(x_h)|^{1/2}\, , 
\end{equation}
\begin{equation} \label{alphaepsilon}
 \epsilon (h):= \alpha (h) ^{- r(h)}, \, \hbox{
  with  $r(h) :=  1$ for even $|h|$,  and  $r(h):=  2$ for odd $|h|$}\, . 
  \end{equation}  
With this extension, $\alpha, \epsilon $ are now defined on the set ${\cal H}^+$ and $\alpha$ satisfies the multiplicative property\footnote{Remark that  this is  not the case for $\epsilon$.} 
 $\alpha(h^k) = (\alpha(h))^k$. 
Now, the set ${\mathcal H}^+$ is endowed with the size $\epsilon$ (closely related to $\alpha$) and a cost $c$. 

The rational trajectories are classically studied with the denominator size: one associates  with $h \in {\cal H}^+$ the denominator  $q(h)= |h'(0)|^{-1/2}$. The distortion constant $L$ introduced in  Property $(P1)$ of Section 3.1 now relates the two sizes $\epsilon$ and $q$, defined on ${\cal H}^+$,  via the inequality 
$$ L^{-1/2}\,  \epsilon(h)^{1/2} \le  L^{-1/2}\,  \alpha(h)^{-1} \le    q(h)  \le    L^{1/2} \, \alpha (h)^{-1} \le  L^{1/2}\,  \epsilon (h) \, , $$
that also proves  that the set ${\cal P}_N$ of rqi numbers with size $\epsilon (x) \le N$ 
 is  finite,  and satisfies $|{\cal P}_N|  = O(N^2)$.   
 
 \subsection {Cost of moderate growth.}\label{Sect:CMG32}  
 A digit-cost $c: {\mathbb N}^* \rightarrow {\mathbb R}_{\ge 0}$ 
is also a mapping  $c: {\cal H} \rightarrow {\mathbb R}_{\ge 0}$, defined by the equality  $c(h_m):= c(m)$. We next extend $c$  to ${\cal H}^+$ in an additive way and define 
$$ c:{\cal H}^+  \rightarrow \mathbb R_{\ge 0} \quad \hbox{with} \quad   c(h):= \sum_{i=1}^{k}c(h_i)\quad \hbox{ if }  h=h_1\circ h_2 \circ \ldots \circ h_k\, .$$
We  now consider  costs of moderate growth\footnote{ We will explain in Section \ref{Sect:33} why we restrict ourselves  to this class of costs.}.

\begin{definition} A digit-cost $c$ is of  moderate growth  if there exists $(A, B)$ with $A \ge 0, B\ge 0$ and $A+B >0$ for which  
\begin{equation} \label{cmg} c(m) \le  A \log m + B \qquad  \hbox{for all  } m \ge 1\, .\end{equation}
\end{definition}

In the present rqi framework,   a precise comparison between the various generating functions  needs a precise study of these costs  (which was not needed  in previous studies, for instance in \cite{BaVa}).
 In particular, Relation  \eqref{cmg3} between costs and derivatives of inverse branches will be central in the proof  of Proposition \ref{Pro:Z-P}.
 
\begin{lemma} \label{cost}
 {For  a digit-cost $c$ of  moderate growth, there exists a  pair $(K,   d) $ of strictly positive numbers  for which
  \begin{equation} \label{cmg3} e^{c(h)} \le  |h'(x_h)|^{-d}, \quad e^{c(h)} \le  K \,    |h'(0)|^{-d}   \qquad \hbox{ for any $h \in {\cal H}^+$}\, .
  \end{equation}The cost is said to be of exponent $d$. }
    \end{lemma}

\begin{proof} 
For any $h$ of depth 1, of the form $h= h_m$,  the equality   $2 \log m = |\log |h'(0)|| $ holds and  Relation \eqref{cmg} writes
$$  c(h) \le   (A/2)    |\log |h'(0)| | + B \qquad  \hbox{for all  } h \in {\cal H}\, . $$
We first prove that the previous relation extends to any $h \in {\cal H}^+$ as 
  \begin{equation}  \label{cmg2} c(h) \le (A/2) |\log |h'(0)| |+  B |h| \qquad  \hbox{for all  } h\in {\cal H}^+\, ,
  \end{equation}
  where $|h|$ is the depth of the LFT $h$.
One  indeed performs  an easy induction;  when $c = h_1 \circ h_2$, one has 
$$ c(h_1\circ h_2) = c(h_1) + c(h_2) \le  \frac A 2 |\log |h_1'(0)||+ \frac A 2 |\log |h_2'(0)||   + B\, |h_1| + B\, |h_2|\, . $$
Now, for any $x \in [0, 1]$,   the inequality  $|h'(0)| \ge |h'(x)|$ holds, and 
$$  |\log |h_1'(0)||+|\log |h_2'(0)||  \le 
 |\log |h_1'(h_2(0))||  +| \log |h_2'(0) || = | \log | h'(0) ||\, . $$
 Relation \eqref{cmg2} is  now proven, and  implies 
 
 \centerline{$c(h) \le (A/2)  |\log |h'(x_h)| |+  B |h| $.}
 
Using  Property $(P2)$, with $ A:=  2a, B := -b \log  \rho$, now  entails   the bound 
  $$ e^{c(h)} \le  \exp[-a \log |h'(x_h)|]  \exp[  B |h|]  = |h'(x_h)|^{-a} \cdot \rho^{-b|h|} \le  |h'(x_h)|^{-(a+b)} \, .$$
  With the distortion constant $L$,  
   this  entails the second inequality with $K= L^{a+b}$.

\end{proof}

 \subsection {\bf  Generating functions of interest.} \label {vargenfunct}
 When the  set $\mathcal{P}$ of rqi numbers is endowed with the size $\epsilon$ and an additive  cost $C$ that comes from a digit-cost $c$,     its basic generating function   is a  Dirichlet weighted generating function  of the form
 \begin{equation}\label{Eq:P}
  P(s,w):=\sum_{x\in \mathcal P}\exp(wC(x))\, \epsilon(x)^{-s} \, , 
\end{equation}
where the variable $w$ ``marks'' the cost $C$ and the variable $s$ ``marks'' the size $\epsilon$.  

  We  now define two variants of series $P$, namely series $Z$ and series $E$, that are closely related to the initial series $P$ and easier to deal with.  The series $Z(s, w)$ ``avoids the primitivity condition'', whereas the series $E(s, w)$ deals with size $\alpha$ instead of size $\epsilon$.

\smallskip
{\bf \em Series $Z$.} To avoid the primitivity condition, we ``replace'' the series $P(s, w)$ by  the ``total'' series                            
\begin{equation}\label{Eq:Z}
Z(s,w):=\sum_{h\in  \mathcal{H}^+}\exp(wc(h))\, \epsilon(h)^{-s} \, , 
\end{equation}
and the decomposition \eqref {HP} 
 leads to the decomposition  \begin{equation}\label{Eq:ZP1}
   Z(s, w) =  P(s, w) +  A(s, w) \, , 
    \end{equation}
$$\hbox{with} \    A(s, w) :=\sum_{k\ge 2}  P_k(s, w) \  \hbox{ and } \  P_k(s, w) :=  \sum_{h \in {\cal P}^{(k)}}  \exp(wc(h)) \cdot \epsilon(h)^{-s} . $$
 Proposition \ref{Pro:Z-P} will prove  that the  ``main'' term in \eqref{Eq:ZP1} is brought by   $ 
P(s,w)$. 

\smallskip
{\bf \em Series $E$.} We  wish to deal with  size $\alpha$   (instead of size $\epsilon$).  Size $\alpha$ indeed  fulfills  multiplicative properties described in Section \ref{Sect:QIN}.  This is why we introduce the series 
$$ Y(s, w) := \sum_{k \ge 1} Y_k(s, w) \, , $$
 \begin{equation} \label {Ysw}
\hbox{with} \qquad  Y_k(s, w) := \sum_{h \in {\cal H}^k}\exp(wc(h))\, \alpha(h)^{2s}=\sum_{h \in {\cal H}^k}\exp(wc(h))\, |h'(x_h)|^{s} \, . \end{equation}
We  also write $Z(s, w)$ as 
$$ Z(s, w) := \sum_{k \ge 1}  Z_k(s, w), \qquad  Z_k(s, w) := \sum_{h \in {\cal H}^k}\exp(wc(h))\, \epsilon(h)^{-s} \, . $$
The notion of minimal even period leads to  Relation \eqref{alphaepsilon} between $\epsilon$ and $\alpha$. This entails the  following relations   between the series $Z_k(s, w)$ and $Y_k(s, w)$, 
$$Z_k(s, w) = Y_k(s/2,w) \ \hbox{ (for $k$ even)}, \quad Z_k(s, w) = Y_k(s, w) \  \hbox{ (for $k$ odd)} \, ,  $$
  and  leads to the decomposition  
  \begin{equation}\label{Eq:ZE1}
  Z(s,w)=E({s},w)+O(s,w)\ , 
  \end{equation}
   where the even and the odd part of $Z(s,w)$ are related to components $Y_k(s, w)$ as \begin{equation}\label{Esw}E(s,w)= \sum_{\substack{k {\rm \, even}\\ { k\ge 2}}}  Y_{k}({s}/{2}, w)\, , \quad O(s,w)= \sum_{\substack{k {\rm \, odd}\\ { k\ge }1}}  Y_{k}(s, w) \, .
 \end{equation}
 Proposition \ref{Pro:Z-P}  will prove that the  ``main'' term in \eqref{Eq:ZE1} is brought by   $ 
E(s,w)$.

\smallskip
\smallskip
\paragraph{\bf \em Sets $\Gamma(a)$.} 
 Each term of the generating functions involves a product 
$ |h'(x)|^s \exp (w c(h)) $. 
With a cost $c$  (of moderate growth)  of exponent $d$  and   a real $a>0$,  we   associate the sets ${\cal C}_a$  and $\Gamma(a)$,  
 $$  {\cal C}_a := \{ (\sigma_0,  \nu_0) \mid \sigma_0 > 1/2, \,    \nu_0>0, \,  \sigma_0- d  \nu_0 =  a \}\, , $$
 \begin{equation}\label{Eq:Q}
\Gamma(a):= \{(t, w)  \mid \Re t>\sigma_0, \, |\Re w|<\nu_0, \quad  (\sigma_0, \nu_0) \in {\cal C}_a \}\, , 
\end{equation}
that are well-adapted to the study of such products.

In the sequel, we will deal with $\Gamma(a)$ or
some of its subsets of the form ${\cal S} \times{\cal W}$ where ${\cal W}$ is a complex 
neighborhood of $0$ of the form $\{ |w|\ \le \nu_0 \}$ and ${\cal S}$ is a  (part of a) vertical strip of the
form $\{ s\mid \Re t> 1- \delta_0\}$. Such a domain is contained in $\Gamma(a)$ as soon as the inequality $1-\delta_0 -d \nu_0 >a$ holds.

\smallskip
{\bf \em Series $Y$.}   We have introduced  the three series $P, Z, E$, together with series $Y$. The series  $Y$ will be  a useful reference in this context, as we now explain:

\begin{lemma} \label{lemY} Consider the series $X(s, w)$ with $X \in \{ P, Z, E\}$ and a pair $(s, w)$ with  $\nu := \Re w$ and $\sigma:= \Re s$.  
\begin{itemize}

\item[$(a)$] The following  relations hold between $X$ and $Y$, 
$$ |X(s, w)| \le  Y(\sigma /2, |\nu|), \quad |P_k(s, w)| \le   Y((k/2) \sigma, k|\nu|)\, .$$

\item[$(b)$] The following bound holds for each component $Y_k(s, w)$ on each $\Gamma(a)$, and involves the distortion constant $L$, 
\begin{equation} \label{bound1} 
 |Y_k(t,w)|\le
  L \,   \rho^{(a-1) k}\, \quad  \hbox{  for any $k \ge 1$} \, .
  \end{equation} 
  \end{itemize} 
  \end{lemma} 
  
  \begin{proof} \ \ 
  $(a)$ The first bound  is  due to the inequality  $\epsilon (h)^{-\sigma} \le  \alpha(h)^{\sigma} = |h'(x_h)|^{\sigma/2} $ that holds  
 for $\sigma>0$. The second bound uses the multiplicative property  of size $\alpha$, namely $\alpha(h^k)= \alpha (h)^k$. 
 
 $(b)$  
For a cost $c$ of exponent $d$, and  a pair  $(t, w)$  with $\sigma: = \Re t$ and $\nu:= \Re w$,  Lemma \ref{cost}  yields 
\begin{equation} \label{bound1/2}
|Y_k(s, w)| \le Y_k(\sigma, |\nu|) =  \sum_{h \in {\cal H}^k}  e^{|\nu| c(h)} |h'(x_h)|^{\sigma}\le \sum_{h \in {\cal H}^k}  |h'(x_h)|^{\sigma -d|\nu|}\, . 
 \end{equation}
  Moreover,  with  Properties $(P1)$ and $(P2)$,  one has $$ 
\sum_{h \in {\cal H}^k} |h(0)- h(1)|  = 1,   \quad  \sum_{h \in {\cal H}^k} |h'(0)| \le L, 
 \quad |h'(x_h)| \le \rho^k \qquad \hbox {for $ h \in {\cal H}^k$}\, ;  $$
 This entails the  bound
$$ \sum_{h \in {\cal H}^k}   |h'(x_h)|^{a}
 \le \rho^{(a-1)k}L\sum_{h \in {\cal 
H}^k} |h(0)- h(1)|\le L \, \rho^{(a-1) k}\, , 
 $$
 and finally the bound  \eqref{bound1}   when  $(t, w) \in \Gamma(a)$. 
 \end{proof}

\subsection{Analytical equivalence between generating  functions}\label{first-prop}
 
The inequality 
\begin{equation}\label{lowerbound}
 Y_k(1,0)\ge \frac{1}{L}\sum_{h\in \mathcal{H}^k} |h(1)-h(0)|=\frac{1}{L}\quad \hbox{ for any } k\ge 1  
\end{equation}
 entails  that  the series that defines $E(s, w)$ is   divergent  at $(2, 0)$. The point $(2, 0)$ plays an important role, and we  now wish to compare more precisely  the three series of interest $P, Z, E$  near the vertical line $\Re s = 2$. We  first  introduce a definition: 
 
 \begin{definition}\label{Def:AP} Two Dirichlet series $A(s,w)$ and $B(s,w)$  are analytically equivalent if,  for any  neighborhood $\mathcal{W}_0$ of $w=0$ small enough, there are two  reals $\delta_0$ and $\delta_1$ with $\delta_1>\delta_0>0$ for which the following holds, 
\begin{enumerate}

\item[$(a)$] For any $w \in {\cal W}$,   each function  $s \mapsto A(s, w)$, $s\mapsto B(s, w)$    is  analytic  on the half plane 
${\cal P}_0 :=  \{ s \mid  \Re s >2+ \delta_0\}$  and  bounded on   the  set  $ \overline {\cal P}_0 \times \overline {\cal W}_0 $; 

\item[$(b)$]   For any $w \in {\cal W}$, the difference   $D(s,w):=A(s, w)-B(s,w)$   is analytic  on  the vertical strip ${\cal S}_1:= \{ s \mid  |\Re s- 2|<\delta_1 \} $ and bounded 
  on $\overline {\cal S}_1 \times \overline {\cal W}_0$.

\end{enumerate} 
\end{definition}

Next Proposition \ref{Pro:Z-P}    shows that  the three series of interest are  all analytically equivalent.

 \begin{proposition} \label{Pro:Z-P}
 Consider  a cost of moderate growth. Then   
  the three generating functions $P(s, w), Z(s, w)$ and  $E(s, w)$ are analytically equivalent.
\end{proposition}

 \begin{proof}  We use Lemma \ref{lemY} and  the series $Y$ plays an important role in the proof.  
 
   {\em Condition $(a)$.} Consider the neighborhood ${\cal W}_0:= \{w \mid |w| <  \nu_0\}$. The domain ${\cal P}_0 \times {\cal W}_0$ is a subset of $\Gamma(a)$ with $a >1$ as soon as $\delta_0 >2 d \nu_0$. Moreover,  
   as soon as $(s/2, w)$ belongs to $\Gamma(a)$ with $a>1$, Lemma  \ref{lemY} entails that $E(s,w)$, $P(s,w)$ and $Z(s,w)$   all satisfy Condition $(a)$.

 \medskip
{\em Condition $(b)$ for  the difference $Z(s,w)-P(s,w)$.} This difference is described   in \eqref{Eq:ZP1}. 
 For  $k \ge 2$,  we use Lemma \ref{lemY} together with bound  \eqref{bound1/2} that entail the inequality  
$$|P_k(s, w)| \le   P_k(\sigma, \nu)\le Y((k/2) \sigma, k|\nu|)  \le    \sum_{ h \in {\cal H}^+}    |h'(x_h)| ^{(k/2) (\sigma -2d |\nu|)}\, .$$ 
Assume  now that  $(s/2, w) \in \Gamma(a) $ with $a>1/2$. This means that  $\sigma -2d	|\nu|$  is at  least $ \gamma > 1$ and let $\gamma = 1 + 3 \beta$ with $\beta >0$. Then, for any $k \ge 2$, the inequality  
$(k/2) \gamma \ge (k-1) \beta +1$ holds and entails  the  bound
$$ |h'(x_h)|^{(k/2)( \sigma -2d|\nu|)} \le |h'(x_h)|^{ (k-1)  \beta}  \cdot |h'(x_h)|\, , $$
  and thus,  in the same vein as  in the proof of Lemma \ref{lemY}, the bound,   
$$Y((k/2)\sigma, k|\nu|) \le   L \,    \sum_{n \ge 1}  \rho^{n(k-1)  \beta}\, , \qquad 
\hbox {for any $k \ge 2$} \, .$$
 Finally,  as soon as $(s/2, w)\in \Gamma(a)$ with $a>1/2$,  one obtains
\begin{align*}
  |Z(s,w)-P(s,w) | &\le  \sum_{k \ge 2}  Y((k/2)\sigma, k|\nu|) \le    L \sum_{n \ge 1} \sum_{k \ge 1}  \rho^{nk  \beta} \\
&=   L  \sum_{n \ge 1} \frac{ \rho^{n \beta}}{1-  \rho^{n \beta}}
\le    L  \frac  {\rho^\beta}  { (1- \rho^\beta)^2}   \, .
\end{align*}
Now,   consider  ${\cal W}_0 := \{ w \mid |w| < \nu_0 \}$.  The domain ${\cal S}_1 \times {\cal W}_0$ is a subset of $\Gamma(a) $ for $a >1/2$ as soon as  $\delta_1 < 1-2 d \nu_0$. If furthermore $\nu_0 <1/(4d) $   then we can choose  $ \delta_1 >\delta_0 > 2 d \nu_0$,    and  
 Condition $(b)$ holds  for  $Z(s,w)-P(s,w)$  for    ${\cal W}_0 $. 

\medskip  
 {\em Condition $(b)$ for  the difference $Z(s,w)-E(s,w)$.} This  difference $Z(s,w)-E(s,w)$ is equal to  the odd series $O(s,w)$. The bound \eqref{bound1} implies that $O(s,w)$ is absolutely convergent and uniformly bounded when $(s/2,w)\in \Gamma(a)$ with $a>1/2$.  Then, for any  neighborhood ${\cal W}_0$ with $\nu_0 < 1/(4d)$,  Condition $(b)$ holds for the same choice of $\delta_1, \delta_0$  as before.  
\end{proof}

Then, with the previous proposition, it will be sufficient to study $E(s,w)$ around $\Re s=2$.  Next  Section \ref{Sect:3} will be then  devoted to this task.

\section{Generating functions and  weighted transfer operators}\label{Sect:3}

Beginning with the initial generating function  $P(s, w)$, we  have  related $P(s, w) $ to {
$E(s, w)$.  In order to study  more deeply the series $E(s, w)$,  
  we introduce our main tool,  the  weighted transfer  operator ${\bf H}_{t, w}$, and further relate  the series  $E(s, w)$  to the transfer  weighted operator ${\bf H}_{t, w}$ and its iterates ${\bf H}^k_{t, w}$ with $t=s/2$.

In Section \ref{Sect:WTO}, 
we 
 present the transfer operators  and   make precise the two functional spaces on which they act.  We describe  the 
first properties of these operators in Section \ref{Sect:33}.   Then Section \ref{gen-pres} explains the role of each functional space. The space $\A$  is a space of analytic functions where the traces of operators is well defined; we obtain a {\em precise}  alternative expression  of the series $E(s, w)$ in terms of traces in Section \ref{Sect:traces};  we then deduce precise information on $E(s, w)$  when $|\Im s|$ is bounded (see Sections \ref{Sec:near} and \ref{Sec:Middle}) and, in particular,  a precise study of its poles in Section  \ref{Sec:near}.    However, the norms of the operators --when they act on $\A$-- exponentially grow with respect to  $|\Im s|$; this is why the space $\A$ is not well adapted to the study of $E(s, w)$ for large  $|\Im s|$.  In Section   \ref{Sec:Dolgo}, we then deal  with the space ${\cal C}^1 ({\cal I})$, where  we use  only {\em approximate } relations between the series $E(s, w)$ and the norms of  transfer operators.  Estimates \`a la Dolgopyat  -- of polynomial growth with respect to $\Im s$--  are valid on  the norms of the operators  on the functional space ${\cal C}^1 ({\cal I})$ and  they entail  that $E(s, w)$ is of polynomial growth for $|\Im s| \to \infty $.
Finally,   with Proposition \ref{Pro:Z-E},  we   will  return from  $E(s,w)$   to $P(s,w)$, and obtain a precise knowlege of the initial series $P(s, w)$  for any complex $s$ (and $w$ near 0) in Section  \ref{Sec:3P}.

\subsection{Weighted transfer operator and functional spaces}\label{Sect:WTO} 
We now introduce our main object,  the weighted transfer operators (via the component transfer operators).

The {\em component weighted transfer operator} $  \mathbf{H}_{[h],t, w}$  relative to an inverse branch $h \in {\cal H}^+$  involves the  map $\und h$ defined in $(P3)$, 
\begin{equation}\label{Eq:H}
   \mathbf{H}_{[h], t, w}[f] (x):=  \exp (w c(h)) \, \und h^t (z) \, f \circ h(z) \, .
\end{equation} 
The sum of  the weighted component operators  ${\bf H}_{[h],t,w}$ taken over the set ${\cal H}$ defines the weighted transfer operator
 \begin{equation} \label {H}
\mathbf{H}_{t, w}:=\sum_{h\in \mathcal{H}} 
{\bf H}_{[h],t,w}\, .\end{equation}
This is the main dynamical object of the study, as an extension of the density transformer of the system, obtained  when $(t, w) = (1, 0)$. \\
Due to the additivity of the cost,  and the multiplicativity of the derivative (or its associated mapping $\und h$),  the $k$-th iterate of the operator  ${\bf H}_{t,w}$ has exactly the same expression as the operator itself, with  a sum  that is  now taken over    $\mathcal{H}^k$,
  \begin{equation*}
{\bf H}_{t,w}^k  = \sum_{h \in \mathcal{H} ^k} {\bf H}_{[h], t, w}\, .\end {equation*}
 As in \cite{BaVa}, the quasi-inverse  plays a central role.  Here, we  will deal with the even quasi-inverse
 \begin{equation} 
 \label{Eq:EO} {\bf E}_{t, w} :=\sum_{\substack{k {\rm \, even}\\ { k\ge 2}}} {\bf H}^k_{t,w}= {\bf H}^2_{t,w} (I- {\bf H}^2_{t, w})^{-1}\, .  
 \end{equation}

\medskip
In the present study,  there are two useful functional spaces  on which these operators act, that we now describe. 

\smallskip
\paragraph {\bf \em The space $\A$.}
Generating functions deal with  fixed point $x_h$ of each  branch $h$. As we will see, traces of   operators, {\em as soon as they are well-defined},  provide a  good tool to deal with fixed points.  This is why  we are led to study the operators when they act on spaces of analytic functions.

We consider the disk  ${\cal V}=\{z\in \mathbb C:\, |z-1|<5/4\}$   described in Property $(P3)$ of Section \ref{rhoUNI} and we deal with  the  Banach  space 
  \begin{equation} \label{calA}
 \A:= \{ f: \overline {\cal V} \rightarrow {\mathbb C} \mid f {\rm \  continuous}, \  f|_{\cal V} {\rm \ analytic} \} \,  
 \end{equation} 
 endowed with the sup-norm
 $  ||f||_{\mathcal{V}}:= \sup\{|f(z)| \mid\, z\in \overline {\mathcal{V}}\}.$ This  is a Banach space. With Property $(P3)$, 
we 
associate   with any $h \in {\cal H}^+$ the  function  $\und h$  which is the analytic extension of the function $|h'(x)|$  to the disk ${\cal V}$. 
Then, with  Property $(P3)$ of Section \ref{rhoUNI},  
 the  component weighted transfer operator $  \mathbf{H}_{[h],t, w}$ 
acts on the Banach  space $\A$. Moreover, the equality that holds for  $(t, w)$, with $t = \sigma + i \tau$ and $\nu =  \Re w$, 
\begin{equation} \label{bad}
| \mathbf{H}_{[h], t, w}[f] (z)| =   \exp (\nu  c(h)) \,  |\und h (z)|^\sigma \,  \exp ( \tau \arg \und h(z))\, f\circ h(z)
\end{equation}
shows that   the norm $||{\bf H}_{[h],t, w}||_{\cal V}$  is of exponential growth with respect to $\tau:= \Im t$.

\smallskip
\paragraph{\bf \em The space $\mathcal{C}^1(\mathcal{I})$.}
This is why   $\A$ is not adapted to our study  when  $\tau:= \Im t$ is large, where we expect a polynomial growth with respect to $\tau$.  However, the exponential term $\exp ( \tau \arg \und h(z))$
disappears when $z$ is real. This is why we  also deal   with the Banach space  $\mathcal{C}^1(\mathcal{I})$ of continuously differentiable functions on the unit interval $\mathcal{I}$ with the usual norm $\|f\|_{(1,1)} := \|f\|_0 +  \|f'\|_0$ where  $\|f\|_0 := \sup\{|f(x)| \mid\, x\in {\mathcal{I}}\}$.    In  this  case, 
 the  transfer operator
    ${\bf H}_{[h],t,w}$  involves  the  absolute value of the derivative, satisfies 
\begin{equation} \label{Eq:Hcompo}
{\bf H}_{[h],t,w}[f](x) = e^{w c(h)} \, 
| h'(x)|^t \,  f \circ h(x)\, , 
\end{equation}
and is proven to act on this space.  We will adapt Dolgopyat-Baladi-Vall\'ee results which entail polynomial growth  of the norm  of the operator 
 ${\bf H}_{t, w}$ when $|\Im t|$ becomes large.

 \subsection {First properties  of the operator ${\bf H}_{t, w}$.}\label{Sect:33}

We now  study 
the main properties of the  operator ${\bf H}_{t, w}$, when the pair $(t,w)$ belongs to the sets defined  in \eqref{Eq:Q}, and when it acts either on $\A$
or $\mathcal{C}^1(\mathcal{I})$. We also  study 
the analyticity of the mapping $(t, w) \mapsto {\bf H}_{t, w}$ on each space.

\begin{proposition} \label{AA}

 Consider a digit-cost $c$ of moderate growth 
 and the associated sets $\Gamma(a)$ (defined in \eqref{Eq:Q}).
 The following holds:
 \begin{enumerate} 

\item[$(a)$]  {\rm [Case  $a>1/2$]}.   For any $(t,w)\in \Gamma(a)$,  the operator $\mathbf{H}_{t,w}$ acts  on $\A$ and on $\mathcal{C}^{1}(\mathcal{I})$.  The map $(t,w)\mapsto \mathbf{H}_{t,w}$ is analytic on  $\Gamma(a)$.

\smallskip
\item[$(b)$] {\rm [Case  $a>1$]}.   For any $(t, w) \in \Gamma(a)$,    the  quasi-inverse $(I-{\mathbf H}_{t,w}) ^{-1}$ acts  on $\A$ and on 
$\mathcal{C}^1(\mathcal{I})$. The map $(t,w)\mapsto (I-{\mathbf H}_{t,w}) ^{-1}$ is analytic on  $\Gamma(a)$.

\smallskip 
\item[$(c)$] {\rm [Useful bounds]}. The  bounds \eqref{Hnbound1/2} and \eqref{Hnbound1}  hold on $\Gamma(a)$  for any $a >1/2$.

\end{enumerate}

\end{proposition}

 \begin{proof} 
With  Lemma 2, 
 the following inequality holds
 \begin{equation}  \label{Hnbound1/2}
{\bf H}^n_{ \sigma,   \nu}[{\bf 1}](-1/4)=  \sum_{h \in {\cal H}^n}  e^{|\nu| c(h)} |h'(-1/4)|^{\sigma}\le   K^{ |\nu|} L^{\sigma}  \sum_{h \in {\cal H}^n} |h'(0)|^{\sigma-d |\nu|}\, .
\end{equation}
  With   $z\in \mathcal{V}$,  and  $f \in  \A$,  the  bound
  $$ \left|{\bf H}^n_{t,w }[{f}](z)\right|\le e^{\pi |\tau| }\, {\bf H}^n_{ \sigma,\nu}[{\bf 1}](-1/4) \, ||f||_{\cal V}$$
  holds, and,    with \eqref{Hnbound1/2}, this
 ensures that, for any $(t,w) \in \Gamma(a)$ with $a>1/2$,  the operator  ${\mathbf H}_{t,w}$ acts  on $\A$;  similar arguments, together with Lasota-Yorke bounds (see \cite{BaVa}),  show that ${\bf H}_{t,w}$ acts  on 
$\mathcal{C}^1(\mathcal{I})$.

Moreover,  the bound obtained in  \eqref{Hnbound1/2} entails  the inequality
\begin{equation} \label{Hnbound1} {\bf H}^n_{ \sigma,   \nu}[{\bf 1}](-1/4) 
\le  K^{ |\nu|} \, L^{a+ \sigma} \,   \rho^{(a-1) n}\, \ \hbox{ for any $(t,w) \in \Gamma(a)$, \  for any $n \ge 1$} \, .
\end{equation}
 Then,together with  Property $(P1)$,  and for   $(t, w) \in \Gamma(a)$ with $a>1$, the  quasi-inverse  $(I-{\mathbf H}_{t,w}) ^{-1}$,  and its variants,  acts  on $\A$ and on 
$\mathcal{C}^1(\mathcal{I})$.
\end{proof}

We  now explain  why we restrict ourselves to costs of moderate growth. 
Consider a cost $c$  that is  {\em not} of moderate growth. In this case,  the ratio $c(m) /\log m$ is not bounded, and,   for any $\nu >0$, and any $\epsilon >0$, 
the inequality   $\nu c(m) >(1+\epsilon) \log m$  holds  for a subsequence of integers $m$; the series  ${\bf H}_{1,\nu}[1](0) $  is thus  divergent   for  any $\nu >0$. Then, there does not exist 
any neighborhood ${\cal S}\times \mathcal{W}$ of $(1,0)$  on which the operator ${\bf H}_{t,w}$ is  well-defined.   
However,  the existence of such  a  neighborhood is crucial  in  all our study,  notably 
for the Implicit Function Theorem used in Theorem \ref{Thm:mainconstants} $(iii)$, 
Dolgopyat's bounds  (Theorem \ref{C}) and Quasi-Powers Theorem (Theorem \ref{E}).

\subsection{Role of each functional space} \label{gen-pres}

 The expression of $\alpha(h)$ given in (\ref{alpha})  in terms of derivatives entails  an alternative  expression for $Y_k(t,w)$ (introduced in Eq.\,\eqref{Ysw}),  that  involves the component  operators   
 ${\bf H}_{[h],t,w}$  defined in \eqref{Eq:Hcompo} as 
 $$Y_k(t,w) = \sum_{h \in {\cal H}^k}  {\bf H}_{[h],t,w}[1](x_h)\, .$$
Here,  each component operator 
is evaluated  at a point  that depends on  $h$, namely the fixed point $x_h$ of $h$. This is the main difference with previous studies (rational trajectories for instance)  where  the same ``$x$''  is used for each component  operator: there, the analog  of   $Y_k(t, w)$ is expressed with the $k$-th  iterate of ${\bf H}_{t,w}$. 

 \smallskip
\paragraph 
{\bf \em   Far from the real axis.}  In Section \ref{Sec:Dolgo}, we  deal with  the  space  ${\cal C}^1 ({\cal I})$, where  there are  interesting Dolgopyat-Baladi-Vall\'ee  bounds    on the $k$-iterates of the operator for  large $|\Im s| $ (established by Dolgopyat in \cite{Do}, and extended by Baladi and Vall\'ee in \cite{BaVa}). We   thus need to  relate   (but not exactly)
\begin{align*}
 Y_k(s,w) \quad \hbox{ and  various} \quad \mathbf{H}^\ell_{s,w}[f](x) \quad  \hbox{for  $\ell \le k$, and some pairs $(f,x)$}\, ,
\end{align*}
 when the operator  acts on the functional space ${\cal C}^1({\cal I})$.

\smallskip 
\paragraph 
{\bf \em  Near the real axis, or in the intermediate region.}  Here,  we need exact alternative expressions, and  we  deal with the space $\A$ defined in \eqref{calA}. In this space, as we will show  in the next Section,   the trace 
of each   weighted   component transfer operator is well-defined and admits expression in terms of $\alpha(h)$. Due to Proposition \ref{Pro:Z-P} that relates $P(s,w)$ and $E(s,w)$,  we  are then led to  introduce the even  quasi-inverse ${\bf E}_{t, w}$ (with $t=s/2$),  already described in \eqref{Eq:EO}, together with its trace.  
  We then obtain  an alternative expression for the series $E(s, w)$,  given in Proposition  \ref{proptraces1}. 
 This will be useful to  find the main singularities of $E(s, w)$ near the real axis, described in  Proposition \ref{Pro:Z-E}  of Section \ref{Sec:near} and  study $E(s, w)$ in the intermediary region (see  Proposition  \ref{Prop:mid} in Section \ref{Sec:Middle}).

     \subsection{Trace of the even quasi-inverse.} \label {Sect:traces}
   
 In \cite{Gr1}, Grothendieck proved   that   the usual formulae that hold for the trace of a matrix are  also valid in a precise framework that is described in the Annex. As it is proven there, the space  $(\A,\|\,\|_{\cal V})$, defined in \eqref{calA} provides an instance of such a framework. 

First, as it is proven in Proposition \ref{proptraces}$(a)$, 
the trace of each component operator ${\bf H}_{[h],t,w}$ is well defined there, and  is expressed in terms of $\alpha(h)=|h'(x_h)|^{1/2}$ and  depth $|h|$ of the branch $h$,   
$${\rm Tr\, }\mathbf{H}_{[h],t,w}= {\alpha(h)^{2t}}\frac{  \exp(wc(h))} {1-(-1)^{|h|} \alpha(h)^2}\, .$$
 Then,  each component of the series $E(s, w)$, relative to a branch $h$ of even depth,   can be written as  
$$ \alpha(h)^{2t} \exp(wc(h)) =  {\rm Tr\, }\mathbf{H}_{[h],t,w} -    {\rm Tr\, }\mathbf{H}_{[h], t +1,w}\, .$$
Second, for $(t,w)\in\Gamma(a)$ with $a>1$,   and as it is proven in Proposition \ref{proptraces}$(c)$,  the trace of the even quasi-inverse ${\bf E}_{t,w}$ defined in \eqref{Eq:EO}   is also well-defined and equals the  sum of the traces  ${\rm Tr\, }{\bf H}_{[h],t,w}$ taken  over all the inverse branches $h\in \mathcal{H}$ with  even depth.  
Then,  introducing  
the Dirichlet series $\Ed(s,w)$,  
\begin{equation}\label{Eq:Ed}
\Ed(s,w):={\rm Tr\, }  {\bf E}_{(s/2), w} \, , 
\end{equation}
the  following relation  holds, 
\begin{equation} \label{Eq:ZE} 
E(s, w) = \Ed(s,w)-\Ed(s+2, w) \, .
\end{equation}
 The next proposition  relates the two series $E(s,w)$ and $\Ed(s,w)$.

  \begin{proposition} \label{proptraces1} 
The  Dirichlet series $E(s,w)$ and  $\Ed(s,w)$  are analytically equivalent. 
 \end{proposition}
 
\begin{proof} \ \ 
Proposition \ref{proptraces}  provides in  \eqref{trqi2}   the explicit expression of   $\Ed(s,w)$.   The maximal value for $h \mapsto \alpha(h)$  over ${\cal H}^+$ is obtained for $h= h_{1}$ and equals $\rho^2$ with $\rho$ defined  in Property $(P2)$. Letting   $K:= 1-\rho^2$,  this entails the bound,
$$|\Ed(s,w)|\le (1/K)  Y(\sigma/2, |\nu|), \qquad  \hbox{ for $ \sigma=  \Re s, \nu=  \Re w$} \, .$$

\smallskip {\em Condition $(a)$.} With bound \eqref{bound1}, the series $ \Ed(s,w)$ satisfies Condition $(a)$ when $(s/2, w)$ belongs to $\Gamma(a)$ with $a>1$.

\smallskip
{\em Condition $(b)$.}
The difference $E(s,w)-\Ed(s,w)$ is equal to $\Ed(s+2,w)$, is a shift to the left of $\Ed(s,w)$.  This  proves with  bound \eqref{bound1}, that  
$\Ed(s+2,w)$ is analytic and uniformly bounded as soon as  the pair $((s+2)/2, w) $ belongs to $\Gamma(a)$ with $a>1$. This means  that the pair $(s/2, w) $ belongs to  $\Gamma(a)$ with $a>1/2$. Then, Condition $(b)$  holds  under the same conditions as in Proposition \ref{Pro:Z-P}. 

 \end{proof}

 \subsection{Study of $E(s, w)$ near the real axis.}  \label{Sec:near}
 
 With Proposition \ref{proptraces1}, we are then led  to study  the  trace $\und E(s, w) $ of the even quasi-inverse. 
Proposition \ref{proptraces} proves  that the map 
$(t,w)\mapsto {{\rm Tr\, }}{\bf E}_{t, w}$ is analytic on $\Gamma(a)$ with $a>1$.   
We now study  an analytic continuation of this map on a neighborhood of $(1,0)$. This will be done in the Annex,
notably in Sections \ref{Sec:SpecOp}, \ref{Sec:SpecTr}, \ref{Sec:SpecCo}. 
Section \ref{Sec:SpecOp}   describes the spectral  decomposition of  the operators ${\bf H}_{t, w}$ near the real axis, and relates it to the spectral decomposition of the even quasi-inverse ${\bf E}_{t, w}$.  We recall here the main steps: 

For real  pairs $(t, w) \in \Gamma(a)$ (with $a>1/2$), the operator ${\bf H}_{t, w} $  is compact and satisfies  strong positive properties that  entail the existence of  dominant spectral objects, in the same vein as the Perron-Frobenius  properties. 
For $(t, w) \in \Gamma(a)$, with $a>1/2$, 
the dominant spectral objects defined for real pairs may be  extended with analytic perturbation of the dominant part of the spectrum  when the pair $(t, w)$ is close to a  real pair. More precisely, 
the Annex introduces in  \eqref{Eq:Sp} the  
spectral decomposition
 of the operator ${\mathbf H}_{t,w}$ that holds  when $(t,w)$  is in a neighborhood of $(1,0)$, 
 $$
 \mathbf{H}_{t,w}=\lambda(t,w)\mathbf{P}_{t,w}+\mathbf{N}_{t,w}\, , 
$$ 
 where  $\lambda(t, w)$ is the dominant eigenvalue, and the spectral radius of $\mathbf{N}_{t,w}$ is less than 1. 
This decomposition gives  rise  
to  the spectral decomposition of the even quasi-inverse given in \eqref{decE} 
$${\bf E}_{t,w}= {\bf H}_{t,w}^2 (I-{\bf H}^2_{t,w})^ {-1}=  \frac{\lambda^2(t, w)}{1-\lambda^2(t,w)} \mathbf{P}_{t,w}+ {\bf N}^2_{t, w}(I-\mathbf{N}^2_{t,w})^ {-1} \, .$$

 Now,    Section \ref{Sec:SpecTr}  of the  Annex studies the trace of this quasi-inverse. It  explains  how the previous decomposition   gives rise to the  spectral decomposition  of  ${\rm Tr\, }{\bf E}_{t,w}$ given in \eqref{dectrE}, 
$${\rm Tr \, } {\bf E}_{t, w}=  
 \frac{\lambda^2(t, w)}{1-\lambda^2(t,w)} +  {\rm Tr \, }{\bf N}^2_{t, w}(I-\mathbf{N}^2_{t,w})^ {-1} \, ,$$
  and leads  to the analytic continuation of $(t,w)\mapsto {\bf E}_{t, w}$.  Moreover, Section 
  \ref{Sec:SpecCo} of the Annex describes the constants that are involved in this decomposition. 

Using Proposition  \ref{proptraces1}, we now transfer these results 
to the generating function $E(s, w)$  when $(s, w)$ is close to the point $(2,0)$.

\begin{proposition} \label{Pro:Z-E}   There is a domain
${\cal T}_2  :=  {\cal S}_2 \times {\cal W}_2$,  formed with a neighborhood ${\cal W}_2$ of 0, and a  rectangle ${\cal S}_2 :=  \{s=\sigma+ i\tau \mid  |\Re s-2| < \delta_2, \,  |\tau|<\tau_2 \}$   where the following holds:  

\begin{itemize} 

\item[$(a)$] For each $w \in {\cal  W}_2$, the mapping $ s\mapsto   E(s, w)$  is meromorphic  on ${\cal S}_2$.   

\item[$(b)$]   It   has a unique pole simple on $ {\cal S}_2$, located at $s_w  :=  2 \sigma(w)$ where $\sigma(w) $ is  defined by the equation $\lambda(\sigma(w), w) = 1$, 
$\sigma(0)=1$.  

\item[$(c)$]  The residue  $v(w)$ of $s \mapsto E(s, w)$  at $s =s_w$ equals $   -1/ \lambda'_t (\sigma(w), w)$. 
The product $(s-s_w) E(s, w) $ is bounded on  $\overline {\cal T}_2$. 
\end{itemize}
\end{proposition}

\begin{proof}   
 Sections \ref{Sec:SpecTr} and \ref{Sec:SpecCo}
  describe the analog of   Properties $(a)$, $(b)$ and $(c)$   for ${\rm Tr \, } {\bf E}_{t, w} $ and introduce  the domains ${\cal S}_1$ and ${\cal W}_1$ and their product ${\cal T}_1 :={\cal S}_1 \times {\cal W}_1$.   We  now return from $t = s/2$ to $s$ and  the change $(s/2)\mapsto s$ brings the factor 2 in the residue.  It transforms the domain $ {\cal S}_1$ into its homothetic $ {\cal S}_2:= 2 {\cal S}_1$ and we let ${\cal T}_2 :={\cal S}_2 \times {\cal W}_1$. 
 We now use   the  property   of Section \ref{Sec:SpecCo}  which  entails  that the product $(s-s_w) \, {\rm Tr\, } {\bf E}_{s/2, w}=(s-s_w)\Ed(s,w)$ is bounded on $\overline {\cal T}_2$.  Then,  Properties $(a)$, $(b)$ and $(c)$  are proven for $\Ed(s, w)$.  Finally,  with  Proposition \ref{proptraces1},  we return to $E(s, w)$ and  conclude the proof. 
  \end{proof}

\subsection{In the middle} \label{Sec:Middle}

This section  studies the behavior of $E(s, w)$  when  $s$  is in the intermediary region, and belongs to the union of two  rectangles  $\{s \mid |\Re s -2| < \delta, \tau_2 \le  |\tau| \le  \tau_4\}$.   The bound $\tau_2$  comes from   the domain ${\cal T}_2$ of Proposition  \ref{Pro:Z-E}, and the bound $\tau_4$ will be  provided  later on by  the domain ${\cal T}_4$ of  Theorem  \ref{Thm:Far} described in the next section.  
The behavior   of the mapping $(t,w)\mapsto {\rm Tr\, } {\bf E}_{t, w}$ in the intermediary region is related to the  behavior of the  even quasi-inverse ${\bf E}_{t, w}$ there,  which is itself   based on the following  spectral property of the plain operator ${\bf H}_t := {\bf H}_{t, 0}$  when $t$ belongs to the vertical line $\Re s = 1$ (see, for instance, the papers by Faivre \cite{Fa}  or Vall\'ee \cite{Va98}).

\begin{classicalProp}{Proposition}\label{D}
 {\em
 On the  punctured vertical line   $\{  \Re t = 1, t \not = 0\}$,  the spectral radius $R(t)$ of the   operator $\mathbf{H}_t$   (when acting on $\A$) is strictly less than 1. }
\end{classicalProp}

  Using analytic  perturbation of parameters $(t, w)$,  the following  result   proves  that    $\Ed(s,w)$, and thus $E(s,w)$, remains bounded in the intermediary region. 
 
\begin{proposition}
 \label{Prop:mid}
 For any pair $(\tau_2, \tau_4)$   with $0<\tau_2<\tau_4$, there exists a domain ${\cal T}_3 := {\cal S}_3 \times {\cal W}_3$ formed with 
 a neighborhood $\mathcal{W}_3$ of $w=0$,  and  a union ${\cal S}_3$ of two rectangles,  
 ${\cal S}_3 := \{s \mid   |\Re s -2|\le \delta_3,\, \tau_2\le |\Im s|\le \tau_4 \} $  together with  a bound  $M_3>0$,   for which  the series   
 $E(s, w)$ satisfies $|E(s,w)|\le M_3$ on the 
domain $\overline {\cal T}_3 $. 
\end{proposition}

\begin{proof}  We first study the function  $(t, w) \mapsto  {\rm Tr}\, \mathbf{E}_{t,w}. $ 
Proposition \ref{D} together perturbation theory of finite parts of the spectrum imply that,   for any two fixed positive numbers,  say   $0 < \tau_2< \tau_4$, there is a   domain ${\cal T}_3:={\cal R}_3 \times {\cal W}_3$,  built with  a neighborhood ${\cal W}_3$ of $w = 0$ and the union ${\cal R}_3 $ of two rectangles of the form 

\vskip 0.1cm
\centerline{${\cal R}_3:= \{ t \mid |\Re t -1| \le ( \delta_3/2) ,  (\tau_2/2) \le |\tau |  \le  (\tau_4/2) \}$}
  so that  the spectral radius $R(
t,w)$  of the operator ${\bf H}_{t, w}$ is strictly less than 1 on the closure of $\overline{\cal T}_3$. 

We may suppose that $\overline {\cal T}_3 \subset \Gamma(a)$ for $a>3/4$. Then, on $\overline {\cal T}_3 $, Proposition \ref{proptraces} $(c)$ proves that the trace ${\rm Tr} \, {\bf E}_{t, w}$ is given by
the absolutely convergent series of general term ${\rm Tr\, } {\bf H}^k_{t, w}.$
Now, in Proposition \ref{proptraces} $(b)$, it is proved that each map $(t,w)\mapsto {\rm Tr\, } {\bf H}^k_{t, w}$ is an analytic map on $\Gamma(a)$ with $a>3/4$. Then, on the open domain ${\cal T}_3$, the map $(t,w)\mapsto {\rm Tr} \, {\bf E}_{t, w}$
is analytic, and it is bounded on  $\overline {\cal T}_3$. 

An  homothety  by a factor of $2$ provides  a neighborhood ${\cal S}_3 := 2 {\cal R}_3$   
so that  $\Ed (s,w)$    is   bounded on  $\overline {\cal T}_3 := \overline{\cal S}_3 \times \overline  {\cal W}_3$ and analytic on its interior. This neighborhood $\overline {\cal T}_3 $ can be chosen so that 
${\cal T}_3\subset \mathcal{S}_0\times\mathcal{W}_0$ from Proposition \ref{proptraces1} and then, 
$E(s,w)$ is also bounded and analytic on ${\cal T}_3$.

  \end{proof}

\subsection {Far from the real axis}\label{Sec:Dolgo}

We  now  study  the series $E(s, w)$  when $w$ is close to 0, and $s$ in a vertical strip  close to the vertical line $\Re s = 2$ but now with large   values of $|\Im s|$.  
We wish to exhibit a polynomial growth of the mapping $s\mapsto E(s, w)$  for  large $|\Im s|$.   We have already explained  in Section \ref{Sect:WTO} why the space $\A$ is not adapted to this task. We thus  {\em do not} use the decomposition  \eqref{Eq:ZE} and directly deal with $E(s, w)$.  We will then compare   directly the series 
$E(s, w)$ and  iterates ${\bf H}^k_{t, w}$  of the transfer operator when acting on ${\cal C}^1({\cal I})$,  for $t = s/2$  with a real part  close to $1$. 

\smallskip
The behavior  of  the iterates   ${\bf H}^k_{t}$ of the (unweighted) transfer operator  $\mathbf{H}_{t}$  in a vertical strip near $\Re t = 1$  was studied by Dolgopyat for  a Markovian dynamical system with a  finite number of branches. When such a system   satisfies the UNI  Property  (recalled  as  Property $(P5)$ in Section \ref{rhoUNI}), Dolgopyat exhibits bounds on the    iterates   ${\bf H}^k_{t}$, which prove that the (plain) quasi-inverse $(I-{\bf H}_{t})^{-1}$ is analytic with polynomial growth.    In \cite{BaVa}, Baladi and Vall\'ee extended his work to dynamical systems with an infinite number of branches, and to weighted operators (associated with a cost of moderate growth).  Dolgopyat dealt with the  Banach space 
${\cal C}^1({\cal I})$,  
 introduced a family of norms $\|\ \|_{(1,\tau)}$ on the space $\mathcal{ C}^1( \mathcal {I} )$ indexed by the real parameter  $\tau\not = 0$,
 $$\|f\|_{(1,\tau)  }:= \|f\|_0 + \frac 1 {|\tau|}  \|f'\|_0 \, , $$ and  obtained estimates on the norms
$\|{\bf H}^k_{t,w}\|_{(1,\tau)}$   for  {
$\tau = \Im t$}. 
The next theorem  describes  the Dolgopyat-Baladi-Vall\'ee estimates, as they are stated in \cite{BaVa}. 

\begin{classicalThm}{Theorem}{\rm  [Dolgopyat-Baladi-Vall\'ee estimates].}\label{C}
 {\em 
  There exist a   domain ${\cal T}_5 :={\cal S}_5 \times {\cal W}_5 $ formed with a neighborhood ${\cal W}_5$ of $w=0$,  and a unbounded rectangle $ {\cal S}_5:=   \{s=\sigma +i\tau  \mid    |\sigma-1| \le \delta_5, \,  |\tau|> \tau_5 \}$,   an exponent $\xi_5 >0$,  a contraction constant $\gamma_5 <1$, and a bound $M_5$, such that, for any $(s, w)\in {\cal T}_5$, the norm $(1, \tau)$ of the $k$-th iterate of the operator ${\bf H}_{t,w}$ satisfies, for {
  $\tau = \Im t$},  
  \begin{equation} \label{Dolgo} 
 \|\mathbf{H}_{t,w}^k\|_{(1,\tau)}\le M_5 \cdot |\tau|^{\xi_5}\cdot \gamma_5^k, \qquad \hbox {for $k \ge 1$} \, .
 \end{equation}
 For any  $w\in \mathcal{W}_5$, the quasi-inverse $t \mapsto (I-\mathbf{H}_{t,w})^{-1}$   is analytic  on ${\cal S}_5$ and has polynomial growth for $|\Im t | \to \infty$. 
 }
\end{classicalThm}

We now study the similar bounds for the series $Y_k(t, w)$ that are the components of the series $E(s, w)$, (with $t= s/2$), 
$$  Y_k(t, w) =   \sum_{h \in {\cal H}^k} \exp (w c(h)) \,  |h'(x_h)|^{t} \, .$$ 
 We  precisely relate $  Y_k(t, w)$  with the transfer operator ${\bf H}^k_{t, w}$, and  
   prove, in the next Theorem \ref{Thm:Far}, that $E (s, w)$  is analytic and of polynomial growth in a vertical strip of the form $|\Re s -2|\le \delta_0$ and $|\Im s |\ge \tau_0$ for  some small enough  $\delta_0 >0$ and large enough $\tau_0>0$.

\begin{theorem} \label{Thm:Far}
There exist a   domain ${\cal T}_4 :={\cal S}_4 \times {\cal W}_4 $ formed with a neighborhood ${\cal W}_4$ of $w=0$,  and a unbounded rectangle $ {\cal S}_4:=  \{s=\sigma +i\tau  \mid   |\sigma-2| \le \delta_4,  |\tau|> \tau_4  \}$ with $\tau_4 \ge 1$,  an exponent $\xi_4 >0$,  a contraction constant $\gamma_4 <1$, and a bound $M_4$, such that, for any 
     $w\in \mathcal{W}_4$,  $ s \mapsto E(s, w)$   is analytic  on ${\cal S}_4$, and  the  inequality  holds, 
   $$|E(s, w)| \le   M_4 \cdot   |\tau|^{\xi_4} \qquad \hbox {for any $(s, w) \in {\cal T}_4$}\, .$$
  \end{theorem}

 \begin{proof}  
 This proof is an extension of the proof given by Pollicott and Sharp \cite {PoSh}.  The authors of \cite {PoSh}  use  an  idea of  Ruelle and relate  the zeta series to the transfer operator,  with two restrictions:  they  only  consider the  unweighted case  where the dynamical system has a finite number of branches.  We here show that  their proof  extends to a weighted transfer operator with an infinite denumerable set of branches and then relate the weighted zeta series to the weighted transfer operator.   
 
 We consider the Dolgopyat domain ${\cal T}_5 = {\cal S}_5 \times {\cal W}_5$ of Theorem \ref{C},    together with a  pair $(s, w)$ with $w \in {\cal W}_5$  and $s$ which belongs to the homothetic $2{\cal S}_5$.    When $s$ belongs to the homothetic $2{\cal S}_5$, the complex $t= s/2$ belongs to ${\cal S}_5$, and we may apply the Dolgopyat bounds to the iterates ${\bf H}^k_{t, w}$. 

\smallskip
{\bf \em General strategy.} Due to the definition of  $E(s,w)$ 
given in  \eqref{Esw}, we deal with $Y_k(t,w)$ with $t = s/2$.     
We let $\sigma := \Re t =  (1/2) \, \Re s;\,  \tau :=  (1/2)\,  \Im s; \, \nu := \Re w$.

We associate  with each  pair  $(h, \ell)$ with $|h|  \le \ell$,  the function 
 \begin{equation} \label{Fhell}
    F_{h} ^{(\ell)} :=  {\bf H}^{\ell}_{t,w} [ {\bf 1}_{{\cal I}_h}] = \sum_{g \in {\cal H}^\ell}  {\bf H}_{[g], t,w}    [{\bf 1}_{{\cal I}_h}] \, .
   \end{equation}
  Denote by $m$ the depth $|h|$ of $h$.   As ${\bf H}_{[g], t, w}$ is a component operator which deals with the LFT $g$, 
    the function ${\bf H}_{[g], t, w}    [   {\bf 1}_{{\cal I}_h}]$  involves  the function $  {\bf 1}_{{\cal I}_h}\circ g $ and 
 there are two cases according  as  $g$ begins with $h$ or not: 
 \begin{itemize}  

\item[$(i)$] if  there exists $u$ for which $g  = h \circ u$,  then    ${\bf H}_{[g], t,w}[  {\bf 1}_{{\cal I}_h}] = {\bf H}_{[h \circ u],t,w}[{\bf 1} ] $; 

\item[$(ii)$]  if $g$ is not written as $g  = h \circ u $  with $u \in {\cal H}^{\ell-m}$,  then $ {\bf H}_{[g],t, w}[ {\bf 1}_{{\cal I}_h}] = 0$. 
\end{itemize} 
Finally, in all the cases,  this leads to another expression  for any  function    $F_{h} ^{(\ell)}$ defined in \eqref{Fhell} and relative to $h \in {\cal H}^m$, 
\begin{equation} \label{Fhell1}
 F_{h} ^{(\ell)} :=   {\bf H}^{\ell}_{t,w} [ {\bf 1}_{{\cal I}_h}] =   \sum_{u \in {\cal H}^{\ell-m} } {\bf H}_{[h \circ u], t, w}  [ {\bf 1}] =
 {\bf H}^{\ell-m}_{t,w} [ F_{h} ^{(m)}]\, .
 \end{equation}
 This entails, with distorsion property and Lasota-Yorke bounds (see \cite{BaVa}), the  useful inequality 
 \begin{equation} \label {Fhell2}   \sum _{h \in {\cal H}^m} || F_{h} ^{(\ell)} ||_{(1, 1)} \le   {\bf H}^{m}_{\sigma, \nu}[1](-1/4) \, .
 \end{equation}
 
 Our object of interest     $Y_k (t, w)$ is written   in terms of these $F_{h}^{(k)}$,  namely, 
   $$   Y_k (t, w) =  \sum_{h \in {\cal H}^k} F_{h}^{(k)} (x_h)\, .  $$
The main idea (which extends an idea due to Ruelle)  is to  write   
$ Y_k (t, w) $    as a sum
 \begin{equation} \label{Deltam}
  Y_k (t, w)   =     \sum_{h \in {\cal H}} F_{h} ^{(k)} (x_h)   +  \sum_{m = 2}^k   \Delta_m, 
  \ \   \Delta_m:= 
 \sum_{h \in {\cal H}^m} F_{h}^{( k)}(x_h) - \sum_{h \in {\cal H}^{m-1}} F_{h} ^{(k)}(x_h). 
 \end{equation}
 We first study the first term
 ${T_k(t,w):=\sum_{h \in {\cal H}} F_{h} ^{(k)} (x_h)},$ then  each difference $\Delta_m$ contained in the second term.
 
  \smallskip
{\bf \em First term.} 
   In this case,   $F_{h}^{(k)}$  is written as 
   $ F_{h}^{(k)} :=  {\bf H}_{t,w}^{k-1}  [F_{h}^{(1)}] $. Then, we have 
   $$ |T_k(t,w)|\le \sum_{h \in {\cal H}} || F_{h}^{(k)}||_0 \le   || {\bf H}_{t,w}^{k-1}||_0 \sum_{h \in {\cal H}} ||F_{h}^{(1)}||_0 \, .$$ Using the relation between $0$-norm and Dolgopyat norm,  together with the  Dolgopyat bound
\eqref{Dolgo}, and  Inequality \eqref{Fhell2}, we obtain 
$$  |T_k(t,w)|\le \sum_{h \in {\cal H}} || F_{h}^{(k)}||_0||  <\!\! <  |\tau|^{\xi_ 5} \gamma_5^{k-1}  {\bf H}_{\sigma, \nu}[1](-1/4) <\!\! <   |\tau|^{\xi_ 5}\,  \gamma_5^{k} \, .$$

\medskip
 {\bf\em Second Term.}  
    
   We first prove the equality 
  \begin{equation} \label{hundh} 
    F_{b(h)}^{(k)}  =  \sum_{ h  \mid  h = b(h) \circ g}   F_{ h}^{(k)}  \qquad \hbox{for any $ b(h) \in {\cal H}^{m-1}$}  \, .
    \end{equation}
 The fundamental interval  ${\cal I}_{b(h)}$ relative to the LFT $b(h)$ is the disjoint union of the fundamental intervals $  {\cal I}_{b(h) \circ g} $ for $g \in {\cal H}$ and, using two times \eqref{Fhell1},  the sequence of equalities hold: 
 \begin{align*}F_{b(h)}^{(k)}  &= {\bf H}^k_{t, w} [{\bf 1}_{{\cal I}_h}] = \sum_{ g \in {\cal H} } {\bf H}^k_{t, w} [ {\bf 1}_{{\cal I}_{b(h) \circ g}}]
  =  \sum_{ g \in {\cal H} } \sum_{u \in {\cal H}^{k-m} }{\bf H}_{[b(h) \circ g \circ u],t, w}  [{\bf 1}] \\
& =  \sum_{u \in {\cal H}^{k-m}  } \, \sum_{g \in {\cal H} }  {\bf H}_{[b(h) \circ g\circ u],t, w}  [{\bf 1}]  
   =  \sum_{u \in {\cal H}^{k-m}  } \,    \sum_{ h \mid h =  b(h) \circ g} {\bf H}_{[ h \circ u],t, w}  [{\bf 1}] \\
& =   \sum_{ h  \mid  h =  b(h) \circ g}  \    \sum_{u \in {\cal H}^{k-m}  }  {\bf H}_{[ h \circ u],t, w}  [{\bf 1}] =  
  \sum_{ h  \mid  h =   b(h) \circ g}   { \bf H}^k_{t, w} [{\bf 1}_{{\cal I}_h}]  =   \sum_{ h  \mid  h =   b(h) \circ g}   F_{ h}^{(k)}.
  \end{align*}
 The equality \eqref{hundh} is proven. Now, as any $h \in {\cal H}^{m-1}$
 is  the beginning $b(h)$ of  some $h \in {\cal H}^m$,  it easily entails 
 the equality
 \begin{equation} \label{Deltamdif}
   \Delta_m =  \sum_{h \in {\cal H}^{m}} F_{ h}^{(k)} (x_h)  -  \sum_{ h \in {\cal H}^{m-1}} F_{h}^{(k)} (x_h) 
  = \sum_{ h \in {\cal H}^{m}}  \left[ F_{ h}^{(k)} (x_h)  - F_{h}^{(k)} (x_{b(h)})\right] \, .
  \end{equation}
 As  the two points $ x_h $ and $x_{b(h)}$ belong to the same fundamental interval ${\cal I}_{b(h)}$ of depth $m-1$,  whose length is less than $\rho^{m-1}$, 
   the difference $\Delta_m$ admits the bound 
  $$ |\Delta_m| \le   \rho^{m-1}  \cdot \left[ \sum_{h \in {\cal H}^m} ||F_{h}^{(k)}||_1 \right]\, .$$
   Using the expression   of the function $F_{h}^{(k)}$  obtained in \eqref{Fhell1}, 
 we are then led to  evaluate the norm $(1, 1)$ of each  function, for $h \in {\cal H}^m$  
 $$  F_{h}^{(k)} = {\bf H}^{k-m}_{t, w} \circ   {\bf H}_{[h], t,w}[{\bf 1}] \, .$$ Using the relation\footnote{In the paper \cite{PoSh}, the authors seem to have forgotten the factor $|\tau|$....} between the norm $(1, 1)$ and the norm $(1, \tau)$, together with    
 the Dolgopyat bound  \eqref{Dolgo} and Inequality \eqref{Fhell2}, this entails the following  bound for $\Delta_m$,  for any $\gamma_4 \ge \gamma_5$,    
\begin{align*} |\Delta_m| & < \!\!< |\tau|  \cdot  || {\bf H}_{t, w} ^{k-m}||_{(1, \tau)} \cdot   {\bf H}^{m}_{\sigma, \nu}[1](-1/4) \cdot  \rho^{m}\\
 &< \!\!<    |\tau |^{1+ \xi_5} \cdot  \gamma_4^{k-m}\, \cdot \left[  \rho^m {\bf H}^m_{\sigma, \nu}[1](-1/4)\right] \\ &< \!\!<   |\tau |^{1+ \xi_5} \,  \gamma_4 ^{k} \cdot \left[   \rho^m \cdot   {\bf H}^{m}_{\sigma, \nu}[1](-1/4)\cdot \gamma_4^{-m} \right]\, .
\end{align*}
   The bound \eqref{Hnbound1} leads to the inequality 
  $$ \rho ^m \cdot {\bf H}^{m}_{\sigma, \nu}[1](-1/4) < \!\!<  \rho ^m \cdot  \rho^{m(\sigma- 1-d|\nu|)}=  \rho^{m(\sigma-d|\nu|)}\, .$$
  The exponent satisfies $\sigma- d |\nu|  \ge 1- \delta_5 -d \nu_5$.  Then,  choosing $\gamma_4<1$ and $\gamma_4 \ge \gamma_5 $ entails the final bound
  $$    \sum_{m = 2}^k   |\Delta_m|    < \!\!<   k \,|\tau|^{ \xi_4} \cdot   \gamma_4^k\, , \qquad \hbox{with \quad  $\xi_4 = 1 + \xi_5$}\, .$$
  Now, we use   the bound
$ |Y_k(t, w) |\le  |T_k(t, w)| +  \sum_{m = 2}^k   |\Delta_m| $,  we change $\tau$ into $2 \tau$, and the series 
$Y_k(t, w)$ itself satisfies, for some constant $\widetilde M_4$,
{$$ |Y_k(t, w)|  < \!\!< \widetilde M_4\cdot    k  \cdot  |\tau|^{ \xi_4} \cdot   \gamma_4^k\, . $$  This entails  the final bound with 
$M_4:=  \tfrac {\widetilde M_4 }{\gamma_4}   \left(\tfrac 1 {1-\gamma_4}\right)^2  \, .$}
\end{proof}

\subsection{Conclusion of Section 3.   main properties of  the  series $P(s, w)$}\label{Sec:3P}
 We now adjust  various neighborhoods and bounds from    Proposition \ref{Pro:Z-P} in Section \ref{first-prop}, 
Proposition \ref{Pro:Z-E} in Section \ref{Sec:near},  Proposition \ref{Prop:mid} in  Section \ref{Sec:Middle} and finally Theorem \ref{Thm:Far} 
in Section \ref{Sec:Dolgo}.   
 
Each result  provides a domain ${\cal T}_i:= {\cal S}_i \times {\cal W}_i$: The index  $i $ equals $0$ for Proposition \ref{Pro:Z-P}, it equals $i = 2$   
for  Proposition \ref{Pro:Z-E}, it equals $i = 3$ for  Proposition \ref{Prop:mid}, and it equals $i = 4$ for  Theorem \ref{Thm:Far}. Each ${\cal W}_i$ is a neighborhood of $0$ and we  denote by  $  {\cal W}'$  the intersection of the four neighborhoods. Each ${\cal S}_i$ is a vertical strip  or a part of a vertical strip 
of the form $ \{ s \mid |\Re s -2 | < \delta_i \}$. We thus   let $ \delta:= \min (\delta_0, \delta_2, \delta_3, \delta_4)$ and consider the strip ${\cal S}:=  \{ s \mid |\Re s -2 | < \delta\}$.  Then, the horizontal lines $|\tau| = \tau_2$ from  Proposition \ref{Pro:Z-E}, and  $|\tau| = \tau_4$ from  Theorem \ref{Thm:Far} separate the strip ${\cal S}$ in three regions, where  the three results  of Proposition \ref{Pro:Z-P},  Proposition \ref{Prop:mid}, Theorem \ref{Thm:Far}  hold when $w$ belongs to ${\cal W}'$. 

We need a last adjustment  of the neighborhood ${\cal W}'$ to better deal with the rectangle ``near the real axis'', namely  $\{ s \mid |\Re s -2| < \delta, |\tau| \le \tau_2 \} $; we want indeed the set of  poles $\{ s_w \mid w \in  {\cal W}' \}$ to be not too close to the left vertical line $\Re s = 2-  \delta$. As the mapping  $\sigma$ is analytic on ${\cal W}' $, the derivative $|\sigma'(w)|$  is  bounded on $\overline {{\cal W}'}$ (say by some $B>0$). Then,  with the new (and final)  neighborhood ${\cal W}$,  defined as 
$$ {\cal W} :=  {\cal W}' \cap \left\{ |w| \le \frac { \delta}{4B} \right\}  \, , $$
the following holds, for $w \in {\cal W} $,  
$$  |\Re  s_w  - 2|   = 2 |\Re \sigma(w) -1| \le 2 |w| B\le  (1/2)\,  \delta, \qquad  |\Re  s_w - (2-\delta) | \ge (1/2)\,  \delta\, .  $$
 Then we let  $\tau_0 := \tau_4 >1$ and $\xi := \xi_4$. We  then obtain (after adjustment of the  various constants  $M_i$) the final result of this section, which  is clearer when using the following notations.

{\bf Notations.} A width $\delta>0$ and two horizontal lines  $ |\tau | = \tau_0$ (with $\tau_0 >1$)  divide the strip $\{s \mid  |\Re s -2| < \delta \}$ into three domains: 
a  bounded rectangle   ${\cal R}:=  \{ s \mid  |\Re s-2| <\delta,\,  |\tau| \le \tau_0\} $ and   the union    ${\cal U}:=  \{ s \mid  |\Re s-2| <\delta,\,  |\tau| \ge \tau_0\} $ of two  unbounded rectangles.

\begin{theorem} \label{Thm:L}
Consider  the
  bivariate  generating series  $P(s,w)$   associated with a digit-cost $c$   of moderate growth  and defined in \eqref{Eq:P}.   There   exists   
a neighborhood $\mathcal W$ of $w=0$,  a width $\delta>0$,    two horizontal lines  $ |\tau | = \tau_0$ (with $\tau_0>1$),   an exponent $\xi>0 $ and  a   bound $ M$, for  which  the following holds: 

\begin{itemize}

\item[$(a)$]  For  $w \in {\cal  W}$, the mapping $ s\mapsto   P(s, w)$  is meromorphic  on   the half plane  $\{ s \mid  \Re s > 2 -\delta \}$.

\item [$(b)$]
It   has a unique (simple) pole, located inside   the rectangle ${\cal R}$, at  the point $s_w = 2 \sigma(w)$  defined by the equation $\lambda(\sigma(w), w) = 1$, $\sigma(0)=1$. \\
The residue  $v(w)$ of $s \mapsto P(s, w)$  at $s =s_w$ equals $   -1/ \lambda'_t (\sigma(w), w)$. \\
The distance between $s_w$ and the vertical line $\Re s = 2 -\delta$ is at least $(1/2) \delta$. \\
   The   function  $v(s, w) :=  (1/s) (s- s_w) P(s, w) $ satisfies 
$| v(s, w)|  \le M$  on $\overline {\cal R}\times \overline {\cal W}$.

\smallskip
\item[$(c)$]       
The inequality $|P(s, w)| \le M   |\Im s|^\xi$ holds   on $\overline {\cal U}\times \overline {\cal W}$; 

\smallskip
\item[$(d)$] The inequality $|P(s, w)| \le M$   holds   on $\{ s \mid \Re s \ge 2 + \delta \} \times \overline {\cal W}$.

\end{itemize}
\end{theorem}

\section{Extraction of coefficients and obtention of the Gaussian law.}

 The previous section provides a precise description of the analytic properties of the Dirichlet  bivariate generating function $P(s, w)$. We now   return to our probabilistic setting.  Section \ref{Sec:Moment} recalls  the role played by the moment generating function  $M_N(w)$ of cost $C$ on the subset ${\cal P}_N$.  Then, Section \ref{Sec:Landau} describes  asymptotic estimates of this moment generating function   that are  obtained  by ``extracting'' the coefficients of the  generating function  $P(s, w)$:   applying  a uniform version of the Landau Theorem leads to  uniform  estimates 
when $w$ belongs to the   neighborhood ${\cal W}$ defined in  Theorem \ref{Thm:L}. 
  With this uniform estimate of $M_N(w)$ at hands, Section \ref{Sec:QuasiP}  applies the Quasi-Powers Theorem, which  ``transfers'' this (uniform) asymptotic estimate into  an asymptotic Gaussian law, with the speed of convergence stated in Theorem \ref{Thm:mainCF}. The precise study of the constants of interest, performed in Section \ref{Sect:constants},  gives rise to Theorem \ref{Thm:mainconstants}.

\subsection{Using the moment generating function.} \label{Sec:Moment} 
In this paper, we  study  (when $N \to \infty$)  the asymptotic distribution   of the cost $C$ [restricted  to ${\cal P}_N$] via   the sequence of \emph{moment generating functions}
$M_N(w):=\mathbb{E}_N[\exp(wC)]$ where $w$ is   a  complex  number close to 0.
From its definition, the moment generating function $M_N(w)$ is  written as a quotient, namely, 
\begin{equation}\label{Eq:MGF}
 M_N(w) := \mathbb{E}_N[\exp(wC)]=\frac{S^{[C]}_w(N)}{S^{[C]}_0(N)} 
\end{equation}
and involves   the \emph{cumulative sum} $S^{[C]}_w(N)$  of the cost $\exp(wC)$  over ${\cal P}_N$,  
\begin{equation}\label{Eq:CS}
S^{[C]}_w(N)  :=\sum_{x \mid \epsilon(x)\le N }\exp(wC(x))\quad \hbox{  with }  \quad S^{[C]}_0(N) =|\mathcal{P}_N|\, .
\end{equation}

\subsection{Landau Theorem}\label{Sec:Landau}  The Landau Theorem  \cite{Lan} was proved by Landau  in 1924. This    is a  strong  tool 
which gives estimates on  the sum of coefficients of a Dirichlet series, provided the series satisfies  (not too) 
strong analytical properties.  This theorem is not so well-known, and many works (as \cite{BaVa} or \cite {PoSh})  
that use it, do not  deal with a {\em strong} version (that yet exists in the original work of Landau \cite{Lan}, but  is perhaps a little bit hidden).
These authors begin with a weak version and prove the  strong version by hand\footnote{In \cite{BaVa}, 
the authors use the {\em weak} version of the Landau Theorem,  and they have to introduce  
the so-called {\em smoothed} probabilistic model. With this {\em strong} version at hands, the part of their work may be changed and shortened.}.   What we mean by {\em weak} or {\em strong} is related  to the exponent $\xi$: the weak version  only deals   with $\xi  <1$, and very often (as it is the case here)  one  would need a {\em strong} version  which would deal with an exponent $\xi$ that may be larger than 1.  A  proof of the {\em strong} version is precisely  
given  in Roux's thesis  \cite{Ro} and   also available in \cite{BeLhVa16}.
We here  need a ``uniform'' version (with respect to parameter $w$) of this strong version: 

\begin{classicalThm}{Theorem}\label{ThmD}
  {[\rm Landau Theorem -- uniform version]} {\em   If  there is a complex neighborhood  ${\cal W}$ of 0 where the series  $P(s, w)$ satisfies the  ``uniform'' conclusions  of  Theorem \ref{Thm:L}, with an exponent $\xi>0$, and a width $\delta$,  the following holds for any   $N\ge 1$ 
$$ S_w^{[c]} (N) =  \frac{v(w)}{2\sigma(w)} \,  N^{2 \sigma(w)}\,  \left[ 1 + O(N^{-\beta})\right], \qquad  \beta := \frac \delta  {2 (\lfloor\xi\rfloor  +3)}\, , $$ 
where the constant in the $O$--term is uniform for $w \in {\cal W}$. } 
\end{classicalThm}

This leads to a first  result, easily deduced from Theorem \ref{Thm:L} and Theorem \ref{ThmD}:  

\begin{proposition} \label{Pro:Landau}
Consider    the set ${\cal P}$ of rqi numbers, and a cost $C$ on the set ${\cal P}$ associated with a digit-cost $c$   of moderate growth. Denote by $C_N$ the restriction of cost $C$ to the set ${\cal P}_N$ which gathers the rqi numbers $x$ with $\epsilon (x) \le N$. Then,  there exists a neighborhood ${\cal W}$  of 0, on which 
the moment generating function $M_N(w)$  admits  for any $N \ge 1$   a ``uniform quasi-powers'' form
\begin{equation} \label{Qp}
M_N(w)= 
\frac{v(w)}{ \sigma(w) \cdot v(0)}\,  N^{ 2 (\sigma(w) -1)}\, \left[ 1 + O(N^{-\beta})\right],
\end{equation}
which involves  the exponent $\beta$ defined in Theorem \ref{ThmD},   two  analytic  functions, the  function  $w\mapsto \sigma(w)$  defined implicitly by $\lambda(\sigma(w), w) = 1$, with $\sigma(0)=1$, and the residue function  $w \mapsto v(w)$, defined in Theorem \ref{Thm:L}. Moreover,  the constant in the $O$--term is uniform for $w \in {\cal W}$. 
\end{proposition}


\subsection{Quasi-Powers Theorem}\label{Sec:QuasiP}
The following result,
known as the Quasi-Powers Theorem, due to Hwang \cite{Hw96},   shows that  a ``uniform quasi-powers'' expression for the moment generating functions $M_N(w) := E_N[\exp(w C_N)]$ entails an asymptotic Gaussian  law for cost $C_N$.  

\medskip

\begin{classicalThm}{Theorem}\label{E}
 {\rm [Quasi-Powers Theorem]} {\em  Consider  a sequence $ C_N$  of  variables   defined  on probability
spaces $( \mathcal{P}_{N}, \Pr_N)$, and their moment generating functions $M_N(w):=\E_N [\exp(w  C_N)] $.
Suppose that the
functions  $M_N(w)$
are  analytic on a complex
neighborhood  ${\mathcal W}$ of zero,  and each one satisfies there
\begin{equation*}\label{QP}
M_N(w) = \exp [\beta_{N}  U(w) +  V(w)]  \left (1 +
O(\kappa_{N}^{-1}) \right) \, ,
\end{equation*}
where the   $O$--term is uniform on ${\cal W}$. Moreover, the two sequences 
 $\beta_{N}$ and   $ \kappa_{N}$ tend to $\infty$ as $N \to \infty$,  and
$U(w)$,  $V(w)$ are  analytic  on ${\mathcal W }$. Then:  
\begin{itemize} 

\item [$(i)$] The mean and the variance satisfy
$$
 \E_N[ C_N]= \beta_N U'(0) + V'(0) + O(\kappa_N^{-1})\, , $$
 $$ \mathbb V_{N}[   C_N] = \beta_{N} U''(0) + V''(0) + O(\kappa_{N}^{-1}) \, .
$$

\item[$(ii)$] Moreover,  if  $U''(0) \not =
0$,  the distribution of $  C_N$ on $  \mathcal{P}_{N}$ is
asymptotically Gaussian, with speed of convergence $O (\alpha_N)$
 with $\alpha_N =   (\kappa_{N}^{-1} + \beta_{N}^{-1/2})$, 
$$ \Pr _N\left[ x  \mid \frac {C_N(x) - U'(0) \log N} { \sqrt {U''(0)\log N} } \le t \right]= \frac 1 { \sqrt {2\pi}} \int_{-\infty} ^t e^{-u^2 /2} du + O(\alpha_N)\, .$$
\end{itemize} }
\end{classicalThm}

\medskip
 Proposition \ref{Pro:Landau} shows that Theorem \ref{E} may be applied to our setting
with  
$$\beta_N =  \log N, \quad  \kappa_N =  N^{\beta} , \qquad \hbox{ and thus}\quad  \alpha_N = (\log N)^{-1/2}\, ,$$  together with  functions $U, V$ defined as 
 \begin{equation} \label{UV}
 U(w)=2(\sigma(w)-1), \quad \hbox{ and } \quad V(w)= \log \left(\frac{v(w)}{ \sigma(w) \cdot v(0)}\right) \, , 
 \end{equation}
  that  involve the two mappings, the function $w\mapsto  \sigma(w)$ and the  function $w\mapsto v(w)$   defined in 
 Proposition \ref{Pro:Landau}. The  next section  performs a deeper study of these functions near $w = 0$, and proves that the two functions  $U, V$ will be analytic on ${\cal W}$  with $U''(0) \not = 0$. This finishes the proof of our main Theorem \ref{Thm:mainCF}.   
 
 \subsection{Constants.}\label{Sect:constants}
 The function $U$ exactly coincides with its analog  which intervenes in the Gaussian law of the cost  $C$ on rational trajectories and  was studied in  \cite{BaVa}.  One has 
\begin{equation}
U(0) =0,  \qquad U'(0)=2\sigma'(0)=-2\frac{\lambda'_w(1,0)}{\lambda_t'(1,0)}\,.
 \end{equation} and  Baladi and Vall\'ee prove $U''(0) \not = 0$.

 With the expression of $v(w)$    given in  Proposition \ref{Pro:Landau}, one obtains 
 $$V(w) = \log \left(  \frac {\lambda_t'(1,0)}{\lambda_t'(\sigma(w),w)}\right) - \log \sigma(w) \, .$$
The first two derivatives  $V'(0), V''(0)$  then involve the  first  three derivatives of the mapping $(t, w) \mapsto \lambda(t, w)$ at $(t, w) = (1, 0)$.

\smallskip  
The present function $V$ {\em does not coincide} with its analog $V_r$ that is associated with rational trajectories and described in \cite{BaVa}. The function $V_r$  is  expressed in terms of the analog of the residual function $v_r$ as 
$$ V_r(w) =  \log \left(\frac{v_r(w)}{ \sigma(w) \cdot v_r(0)}\right),$$
$$v_r(w)  :=   \frac{-1}{\lambda_t'(\sigma(w),w)}  \frac   {1}{ \sigma(w)} {\bf F}_{\sigma(w) , w}\circ {\bf P}_{\sigma(w), w}[1](0) .$$
   This expression involves the projector  ${\bf P}_{t, w}$ on the dominant eigensubspace together with the ``final'' transfer operator ${\bf F}_{t, w}$ which  only deals  with digits $m \ge 2$, and is defined as
  $$ {\bf  F}_{t, w} [f](x) := {\bf H}_{t, w}[f](x) - \frac {e^{w c(1)}}  {(1+x)^{2t}} f\left( \frac 1 {1+x}\right) \, .$$
  The equality $
   {\bf  F}_{t, w} \circ {\bf P}_{t, w}  [1](0) = f_{t, w}(0) - e^{w c(1)} f_{t, w}(1)$ holds and  involves the dominant eigenfunction $f_{t, w}$. Finally, 
  $$   v_r( w) = \frac{ -1}{ \sigma(w) \, \lambda_t'(\sigma(w),w) } \left( f_{\sigma(w), w}(0) - e^{w c(1)} f_{\sigma(w), w}(1) \right) \, .$$

 Then Theorems \ref{Thm:mainCF} and \ref{Thm:mainconstants} are now  proven, and this ends the main part of our study.

\section{Possible extensions and open problems.}\label{Sect:extensions}

 This section  describes  possible (easy) extensions of the present work.  First,    Section \ref{Sec:levy}   studies  the probabilistic behavior of the ``cost'' $\log \epsilon (x)$. Then, the following two sections (\ref{Sec:bounded1} and \ref{Sec:bounded2}) are devoted to the ``constrained'' cases, where  the numbers have  bounded  digits (in their continued fraction expansion). Finally, Sections 
\ref{Sec:lattice}  and \ref{Sec:nonlattice}  discuss  possible obtention of local limit theorems,  and  study the speed of convergence towards the local limit law 

\subsection{The cost $C(x):= \log \epsilon(x)$}\label{Sec:levy}
Even though  the cost $x \mapsto \log \epsilon (x)$ is {\em not}  an additive cost,   it can be easily studied via our methods. The bivariate generating functions 
 (primitive or general)  involve the general term 
 
 \vskip 0.1cm 
 \centerline{$ e^{w \log \epsilon (x)} \cdot \epsilon(x)^{-s} =  \epsilon(x)^{-(s-w)}$}
and  exactly coincide with the (univariate) generating functions $P(s-w), Z(s-w)$. The  pole function $ \sigma$ is defined as $\sigma (w)  :=  1+ w$ and the residue function 
 $ v$ associates with $w$    the real $  -1/ \lambda'(1+w)$. 
 Then $U'(w) = 2$ for any $w$, and the  second derivative $U''(0)$ equals 0.  Then, Assertion $(i)$ of  Theorem \ref{E}  applies and shows that the variance is  very small,  of order $O(1)$, whereas the mean is of order $\log N$. This means that the cost $C:= \log \epsilon$ is very concentrated around its mean. Assertion $(ii)$ of Theorem \ref{E} does not apply, and there  does not exist a Gaussian limit law. Always with Assertion $(i)$ of Theorem \ref{E}, we obtain precise estimates of the expectation and  the variance  of the cost $x \mapsto \log \epsilon (x)$   on the set ${\cal P}_N$ which gathers the rqi numbers with size $\log \epsilon (x) \le N$. These estimates  involve the first derivatives of the dominant eigenvalue $t\mapsto  \lambda(t)$ at $t = 1$, 
 \begin{equation*}
  \E_N[C]=  2 \log N - \frac {\lambda''(1)}{\lambda'(1)} +O(N^{-\beta}), 
\end{equation*}
\begin{equation*}
{\mathbb V}_N[C]=  \frac { \lambda'''(1) \lambda'(1)  -\lambda''(1) ^2}{ \lambda'(1)^2}+O(N^{-\beta})\  \, .  
\end{equation*}
The dominant term of the mean value is well-known, but the  other terms of the previous estimates appear to be new, notably the  important fact that  the variance is of order $O(1)$.

\subsection{Quadratic irrational numbers with bounded digits (I)}\label{Sec:bounded1}
Here, a given bound $M \ge 2 $ is fixed and possibly equal to $\infty$.    This section is devoted to study  the ``constrained''  set ${\cal P}[M]$  which gathers  rqi numbers whose continued fraction expansion only contains digits  at most equal to $M$. The  analog set ${\cal R}[M]$ of reals whose continued fraction expansion  only contains digits  at most equal to $M$ is extensively studied,   and there are  many works that  deal with  its Hausdorff dimension $\sigma_M$. The  constrained  (unweighted) operator defined as  
\begin{equation} \label{HMs}
 {\bf H}_{M,t} [f] (x) :=  \sum_{m \le M} \frac {1}{(m+x)^{2t}}\,   f\left( \frac 1 {m+x}\right) 
 \end{equation}
directly appears  here, together  with its dominant eigenvalue $t \mapsto \lambda_M(t)$,  
as $\sigma_M$ coincides with the value $t$ for which $\lambda_M(t) = 1$.   The  strict inequality $\sigma_M<1$ holds  as soon as $M <\infty$. 
  
  It is  natural to study     the  rational numbers and the rqi numbers whose continued fractions expansion follows the same constraints, namely the sets  ${\cal Q}[M]:= {\cal Q} \cap {\cal R}[M]$  and ${\cal P}[M]:= {\cal P} \cap {\cal R}[M]$. The rational case was studied by Cusick and Hensley, then  Vall\'ee in \cite{Va98}  provided an unified framework based on analytic combinatorics,  where she analyses, for a fixed bound $M$, the probability  $\pi(N, M)$ that a rational  number  of denominator at most $N$   (or a quadratic number  $x$ with $\epsilon (x) \le N$) has all its digits at most equal to $M$.  In each case,  the  estimate   $$\pi(N, M) = N^{2(\sigma_M-1)} \left[ 1 + \eta_M(N)\right]  \qquad  \hbox{where $\eta_M(N) \to 0$ when  $N \to \infty$} \, ,  $$
is proven to  hold.    Moreover,  she also  studies the  mean value of the depth $p$ (in the rational case) or of the period length  $p$ (in the quadratic irrational  case), and proves   the asymptotic estimates, for any fixed $M$, 
  $$\E_N[p] = \mu_M \cdot  \log N\,  [1 + \eta_M(N)], \  \hbox{with} \  \mu_M :=  \frac {-2}{\lambda_M' (\sigma_M)}, \    \eta_M(N) \to 0 \, . $$
   We note that  the constant $\mu_M$ is the same in the two cases. 
    
  She  deals  in the two cases  with the  ``constrained'' analogs   $Q_M(s)$ and $P_M(s)$ of  univariate  ``unconstrained'' generating functions $Q(s)$ and $P(s)$, namely, 
  $$ Q_M(s) := \sum_{x \in {\cal Q}[M]}  q(x)^{-s}, \qquad P_M(s) :=   \sum_{x \in {\cal P}[M]}  \epsilon(x)^{-s} $$ that she relates to  the constrained  (plain) transfer operator   defined in \eqref{HMs}.
  Her results are based on spectral properties of this constrained operator on the half-plane $\Re s >\sigma_M$,  that are  further transfered to the coefficients of  univariate series $P_M(s), Q_M(s)$. As she was  (only) interested in average-case results, and  did not know (at the moment) the existence of bounds \`a la Dolgopyat, she dealt with Tauberian Theorems, which do not provide  estimates  for the remainder term $\eta_M(N)$ for $N \to \infty$.    
  
  We now  mix her previous approach  with our present methods, and  exhibit a general  ``constrained'' framework  which gathers the present  rqi framework   and the rational  framework of \cite{BaVa}. We introduce the bivariate constrained generating functions associated with some additive cost $C$, namely
  $$ Q_M(s, w) := \sum_{x \in {\cal Q}[M] } e^{w C(x)} \cdot q(x)^{-s}, \qquad   P_M(s, w) :=  \sum_{x \in {\cal P}[M] } e^{w C(x)} \cdot \epsilon(x)^{-s} \, ,$$
together with the weighted constrained operator  
   $${\bf H}_{M,t, w} [f] (x)= \sum_{m \le M} \frac {e^{wc(m)}}{(m+x)^{2t}}\,   f\left( \frac 1 {m+x}\right).$$
   We obtain (easily) the following result: 
 
   \begin{theorem}  With a bound $M \ge 2$, and   a cost  of moderate growth,  associate the additive cost $C$  on each set  ${\cal Q}[M]$ or ${\cal P}[M]$. 
Then,  as soon as $M \ge M_0$ (for some $M_0 \ge 3$),  and on  each of the two  sets,   

-- the set ${\cal Q}_N[M] $ of  rationals $ x \in {\cal Q}[M]$ with denominator $q(x) \le N$  

-- or  the set ${\cal P}_N[M]$ of   quadratic irrationals  $x\in {\cal P}[M]$ with $\epsilon(x) \le N$,  

the  distribution of  $C$     
is asymptotically Gaussian (for $N \to \infty$), and there are two constants   $\mu_M(c), \nu_M(c)$ for which 
$$\Pr _{N, M}\left[ x  \mid \frac {C(x) - \mu_M(c) \log N} { \sqrt {\nu_M(c)\log N} } \le v \right]= \frac 1 { \sqrt {2\pi}} \int_{-\infty} ^v e^{-\frac{u^2}{2}} du + O\left( \frac 1 {\sqrt {\log N}} \right).$$
\end{theorem}   

\begin{proof}  [Sketch of the proof]  We adapt here the general description given in the introduction. Now the Dirichlet series $Q_M(s, w), P_M(s, w)$ have  both a  pole   on the curve $s = 2\sigma_M(w)$ where  $\sigma_M(w) $ is implicitly defined  by the equation  

\centerline{$\lambda_M(\sigma_M(w), w) = 1$, \quad  $\sigma_M(0)= \sigma_M$. }
 Moreover, if $M$ is large enough (say $M \ge M_0$),  the real number  $\sigma_M$ is close enough to 1, and the vertical line $\Re s = \sigma_M $ is contained in the Dolgopyat strip of Theorem \ref{C}. Then,  bounds \`a la Dolgopyat are available for the constrained transfer operator (which can be viewed as a perturbation of the unconstrained one, as  in \cite{CeVa}). This entails the polynomial growth of the quasi-inverse $(I-{\bf H}_{M, t, w})^{-1}$, and  thus the polynomial growth of the two  series $Q_M(s, w), P_M(s, w)$  for $|\Im s| \to \infty$ and $w$ close to 0.  The constants $\mu_M(c)$ and $\nu_M(c)$ are equal to the first and second derivatives of the function $w \mapsto 2 \sigma_M(w)$ at  $w= 0$. 
\end{proof}

  \subsection{Quadratic irrational numbers with bounded digits (II)}\label{Sec:bounded2} Now, the bound $M$ is no longer fixed, and  we wish  to  describe the behavior  of the
   probability $\pi(N, M)$ (both in the rational {\em and} in the  rqi case) when $N$ and $M$  both tend to $\infty$.  First, the   behavior of the Hausdorff dimension $\sigma_M$ of the set ${\cal R}[M]$ has been studied by Hensley in  \cite {He3}, who     exhibits an asymptotic estimate for the speed of convergence   of $\sigma_M$  towards 1, 
  $$  2(  \sigma_M -1) = - \frac {2}{\zeta(2)} \frac 1 M -   \frac {4}{\zeta(2)^2} \frac {\log M}{M^2} + O \left( \frac 1 {M^2}\right)  \qquad (M \to \infty)\, .$$
  
Then,  the authors of the present paper studied in \cite{CeVa} the probability $\pi(N, M)$ in the rational case. Besides some precisions and extensions of a  first result, previously obtained by Hensley in \cite{Hebd}, they  specially develop a methodology (based on  principles of analytic combinatorics) which can be easily transfered from the rational case to the rqi case and implies  the following general result: 

 \begin {theorem} There  are  
 an integer $M_0= M_0\ge 3$, and a real $\beta$, with $0< \beta <1/2$     
so that, for any $N \ge 1, M \ge M_0$,  the  two probabilities  $\pi (N, M)$ that 
\begin{itemize} 
\item [--] a  reduced quadratic irrational  $x$ with a size $\epsilon (x) \le N$
\item[--]   or  a rational number $x$ with  a denominator  $q(x) \le N$ 
\end{itemize}
\vskip -0.1cm 
  has all its digits less than $M$,  satisfy  
$$\pi(N, M)=   N^{2(\sigma_M-1)}  \left[ 1 +O(N^{-\alpha})\right] \cdot \left[1  + O\left( \frac {\log M}{M}\right)\right]  \, . $$
  Here,   $\sigma_M$  is the Hausdorff dimension of the set  ${\cal R}[M]$, and  $\beta$ is   the width of the Dolgopyat strip. 
 \end{theorem}

This result exhibits  a threshold phenomenon, already obtained by Hensley and precisely described in \cite{CeVa}  in the rational case},  depending on the relative order of  $\sigma_{M}- 1$ (of order $O(1/M)$)    with respect to $n:=\log N$, 
\begin{itemize}  
 \item[$(a)$]  If $M/n  \to \infty$, then, almost  any  number of size at most $N$   has all its digits less than $M$.
\item[$(b)$]  If $M/n  \to 0$, then, almost any  number of size at most $N$ has at least one of its digits greater than $M$.
\end{itemize}

\subsection{Local limit laws - case of a lattice cost. } \label{Sec:lattice} 
For any additive cost  $C$ associated with a digit-cost $c$ of moderate growth,   Baladi and Vall\'ee  
also obtained in \cite{BaVa}  a local limit theorem on rational trajectories. Moreover, in the case when the cost $c$ is lattice, they prove an  optimal speed of convergence. We recall that a cost is {\em lattice} if it is non zero and there exists $L \in {\mathbb  R}^+$  so that $c/L$ is integer-valued.
The largest such $L$ is then called the span of $c$. For instance,  the three costs of interest are integer valued  and thus lattice (with span 1).

Local limit theorems  also deal with the moment generating function $ M_N(w) :=\E_N[\exp(wC)]$, but now with  a complex  number $w$ on the vertical line $\Re w = 0$.  When the cost $c$ is lattice with span $L$, it is enough to study the generating function $P(s, w)$ or its analog $Z(s, w)$ for 
a pair $(s, w)$  when $s$ belongs to a vertical strip near $\Re s = 2$ and $w \in ]-i \pi/L, + i \pi/L[$.   
 It is straightforward to adapt  the  proof  of \cite{BaVa} to the present framework and  derive the following result which     is expressed with the following scale, 
  \begin{equation} \label {scaleQ}
    {Q} (y, N):=   \mu(c) \log N + y \, \ \sqrt {\nu(c) \log N}, \qquad  (y \in  {\mathbb R})\, .
    \end{equation}
   
\begin{theorem}{ \rm [Local Limit Theorem for lattice costs.]} For any lattice digit-cost c  of moderate growth and of span $L$, letting $\mu(c) > 0$ and $\nu(c) >0$ be the constants from Theorems \ref{Thm:mainCF} and \ref{Thm:mainconstants}, the following holds  for the cost $C$ associated with $c$ on the set ${\cal P}_N$, 
$${\mathbb P}_N \left[  x \in {\cal P} \mid \left|C(x) -   Q(y, N)\right| \le  L/2   \right]
  = \frac 1 {\sqrt { \nu(c) \log N}}  \frac  {e^{-y^2/2 }}  {\sqrt 2 \pi}   + O \left( \frac 1 {\log N}\right)  $$
  with a $O$ uniform for $y \in {\mathbb   R}$.
\end{theorem}

\subsection{Local limit laws - case of a  non lattice cost. } \label{Sec:nonlattice}
In \cite{BaVa}, the authors did not succeed to estimate the speed of convergence when the cost is non lattice:  this problem is known to be difficult, even  in the simpler case of  dynamical system with affine branches.  The problem was  later considered by Baladi and Hachemi,   then  by Vall\'ee in \cite{Va12}  (always  for rational trajectories) 
who obtained a local limit theorem  with a speed of convergence, that now depends on arithmetical properties of the cost $c$.   More precisely,  when  the cost satisfies a diophantine condition,  then the speed depends on the irrationality exponent that intervenes in the diophantine condition, as we now explain. 

We  first recall that a  real number $y$ is diophantine of exponent $\mu$ if there exists $C > 0$ such that the inequality $|x -  (p/q)| > Cq^{-(2+\mu)}$ is satisfied for any pair $(p, q) \in {\mathbb  N}^2$. We now state the diophantine condition on the cost $c$.   
 
  \begin{definition} For any pair $(h, k) \in {\cal H}^+$  let $c(h, k) := |h| c(k) - |k| c(h)$. The cost $c$ is diophantine of exponent $\mu$
if  there   exists a triple $(h, k, \ell)\in ({\cal H}^{+})^{ 3} $  for which  
 \begin{itemize}  
   \item[$(i)$] $c(h, k)$ and $c(h, \ell)$ are  non zero,  
   
  \item[$(iii)$]    
 the ratio $c(h, k)/c(h, \ell)$  is  diophantine of exponent $\mu$.
 \end{itemize} 
 \end{definition}
 
Then, Vall\'ee  in \cite{Va12} proves that the cost $C$ associated with a digit cost  $c$  (diophantine of exponent $\mu$)  and restricted to the rational numbers of denominator at most  $N$ asymptotically follows a local limit law, where the speed of convergence depends   of the exponent $\mu$.  

\medskip 
\begin{classicalThm}{Theorem}
 {\em 
    Consider  a cost $c$ of moderate growth,  diophantine of exponent  $\mu$, and  let   $\chi :=    6 (\mu +1) $.   Then,   the following holds  for the cost $C$ associated with $c$ on the set ${\cal P}_N$ : for any  $\epsilon $ with $\epsilon <1/\chi$, for any compact interval $J \subset {\mathbb R}$, there exists a constant $M_J$ so that,  for every $y \in {\mathbb R}$ and all integers $N \ge 1$, 
 $$ \left| \sqrt {\log N} \, \Pr_N[   C -{ Q}(y, N) \in J]  -  |J| \frac {e^{-y^2/2}}{   \sqrt {2\nu(c) \pi}}  \right| \le \frac {M_J}{{\log^\epsilon N}}\, .$$}
\end{classicalThm}

\medskip 
The proof  of \cite{Va12}  is based on fine diophantine arguments, and is  closely related to the generating function $Q(s, w)$ of the rational case: it is not clear if this result may be adapted to the generating function $P(s, w)$ of the rqi case.

\section {Annex: Nuclearity and traces of quasi-inverses.} \label{annex}

This annex is devoted to provide detailed proofs of  Propositions \ref{proptraces1} and \ref{Pro:Z-E}. 
The main object here is the trace of various transfer operators. In which sense the usual notion of trace of matrices may be extended to operators  that act on  Banach spaces was  mostly clarified by Grothendieck,  for instance,  in \cite{Gr1}.  Later on, Ruelle and Mayer applied Grothendieck's theory to the study of periodic points in dynamical systems as it is explained in Mayer's survey \cite{Magreen}.   

Most of the  classical studies in the dynamical system context indeed deal   with the Fredholm determinant $F(u)$ associated with a family of  operators ${\bf G}_t$,  
 which satisfies  the trace formula, 
    $$  F(u, t) :=  \det (I- u{\bf G}_t) =   \exp  \left[  - \sum_{k \ge 1}  \frac {u^k}  k\,  { \rm Tr \, } {\bf G}_t^k\right]\, ,$$
   and its logarithmic derivative $t \mapsto \eta(t)$,
   $$  \eta (t)  := \frac{\partial} {\partial u} \log F(u, t)\big|_{u = 1} = \sum_{k \ge 1}   { \rm Tr \, } {\bf G}_t^k \, . $$
 Here, in an approach which appears to be new,   we   directly deal with  this object, which  is proven  to  exactly coincide with the trace of the quasi-inverse $(I-{\bf G}_t)^{-1}$, without studying the Fredholm determinant.

\subsection {Basic facts.} 
\begin{definition} {\rm [Grothendieck]}\label{nuclear} Consider a Banach space ${\cal B}$ that admits a Schauder basis. 
  A linear operator $\mathbf{H}: {\cal B} \rightarrow {\cal B}$ 
  is  
  \emph{nuclear}  if it admits a  representation 
\begin{equation}\label{Eq:representation}
 \mathbf{H}[f]=\sum_{n\ge 1} \lambda_n\, f_n^\star[f] \, f_n
\end{equation}
where $\{f_n\}$ and $\{f_n^{\star}\}$ are families, respectively  in $\mathcal{B}$ and $\mathcal{B}^\star$,  with $\|f_n\|\le 1$ and $\|f_n^\star\|\le 1$, and $\lambda_n$ is a sequence of  complex numbers  that belongs to 
$\ell^1$. \\
The sum  $\sum_{n\ge 1}\lambda_n f_n^\star[f_n]$ does not depend on the choice of the  representation  given  in  \eqref{Eq:representation}, and defines the {\em functional trace} of the operator.\\  If $q := \inf\{p   \mid  (\lambda_n) \in \ell^p\}$ the operator ${\bf H}$ is said to be nuclear of order $q$.
\end{definition}

Remark that  such a nuclear operator  belongs to $\mathcal{L}(\mathcal{B})$ and is also compact.
Its {\em spectral trace} is defined  as the  sum of its eigenvalues counted according to their algebraic multiplicities, if it  exists.  The following  result  obtained  by Grothen\-dieck in  \cite{Gr1} and used in   Mayer's \cite{Magreen} and Momeni-Venkov's  works  \cite{MV}  exhibits important properties of the nuclear operators of order $q \le 2/3$:

\begin{classicalProp}{Theorem}\label{Gro1} {\rm [Grothendieck]}
 {\em The functional trace  of  a   nuclear operator  of order $q\le 2/3$    coincides  with its spectral trace. }
\end{classicalProp}

\subsection{Two useful lemmas.}

 The following two lemmas describe (sufficient) conditions  for dealing with  the trace of a series of operators.

  \begin{lemma}\label{thm-dream}  Consider a Banach space $(\mathcal{B},\|\phantom{x}\|)$ and  a sequence of  operators   ${\bf H}_{[j]}\in {\cal L}( {\cal B})$ for $j \in {\cal J}$.  Assume  there exists $q_0 >0$ for which the two following conditions hold: 
  
  \begin{itemize}  
\item[$(H1)$] Each  operator ${\bf H}_{[j]}$ is nuclear of order $q < q_0$, associated  with  a sequence of coefficients $ m \mapsto \lambda_{m, [j]}$; 
  
\item[  $(H2)$]   The series of general term $\sigma_j:= \sum _{m = 0}^\infty  |\lambda_{m, [j]}|^{q_0} $ is convergent. 

\end{itemize}
 
  Then, the following holds for the operator    $ {\bf H} := \sum_{j\in {\cal J} } {\bf H}_{[j]}$
    
  $(i)$  It  is   nuclear of order $q \le q_0$;  
  
  $(ii)$  If $q_0 \le 2/3$, the equality holds  between the  (spectral) traces,  \ 
   ${\rm Tr\, } {\bf H} = \sum_{j \in {\cal J}}  {\rm Tr\, } {\bf H}_{[j]}$. 
  \end{lemma}

\begin{proof} 
Each operator ${\bf H}_{[j]}$ admits a decomposition of type \eqref{Eq:representation}
 which involves a  sequence $m\mapsto \lambda_{m, [j]}$  that belongs to $\ell^{q_0}$ and   the inequalities $||f_{m, [j]}||  \le 1, \ \    ||f^\star _{m, [j]}|| \le 1$ hold.
Hypothesis     $(H2)$  entails that the sequence $(m, j) \mapsto  \lambda_{m, [j]}$ belongs to $\ell^{q_0}$. Then, the decomposition 
$$ {\bf H} [f]  =  \sum_{(m, j) \in {\mathbb N}\times {\cal J}}  \lambda_{m, [j]}\,   f_{m, [j]}\,   f^\star _{m, [j]} [f]$$
 (with $||f_{m, [j]}||  \le 1$ and $ ||f^\star _{m, [j]}|| \le 1 $)  
 entails that the operator ${\bf H}$ is  nuclear of order $q \le q_0$. 
 Moreover,   for $q_0 \le 2/3$,   the spectral traces satisfy $${\rm Tr\, } {\bf H} = \sum_{m, j} \lambda_{m, [j]}  \,  f^\star _{m, [j]} [f_{m, [j]}] = \sum_{j \in {\cal J} } \sum_{m = 0}^\infty  \lambda_{m, [j]} \,   f^\star _{m, [j]} [f_{m, [j]}]  = \sum_{j \in {\cal J}} {\rm Tr\, } {\bf H}_{[j]}\, .$$
\end{proof} 

The next lemma is useful for dealing with  operators   which involve quasi-inverses.  

 \begin{lemma}\label{useful} Consider an  operator ${\bf T}\in {\cal L}({\cal B})$ and a nuclear operator $\bf A$ of order $q \le 2/3$ 
 together with   its  sequence  $(\sigma_n)$ of coefficients  of  $\ell^1$ norm  equal to $\sigma$.    Then: 
 
 $(i)$  The  operator ${\bf AT}$ is nuclear  of order $q \le 2/3$
 and the inequality   holds: \begin{equation}\label{inequseful}
 |{\rm Tr\, } {\bf AT}  | \le \sigma \cdot \|{\bf T}\|\, .
\end{equation}

$(ii)$
 Assume  that the  spectral radius of $T$ is strictly less than 1. 
 Then, the operator $(I-{\bf T})^{-1}$ exists in ${\cal L}({\cal B})$,  the operator ${\bf A} (I- {\bf T})^{-1}$   is nuclear of order $q \le 2/3$ and the following equality holds  on the traces, 
 $$ {\rm Tr\, }   {\bf A} (I- {\bf T})^{-1} =  {\rm Tr\, }    \sum_{k = 1}^\infty {\bf A} {\bf T}^k   =   \sum_{k = 1}^\infty  {\rm Tr\, } {\bf A} {\bf T}^k\, .$$
\end{lemma}
\begin {proof}   Assertion $(i)$ is clear.  For $(ii)$: 
 The operator ${\bf A} (I- {\bf T})^{-1}$ is nuclear of order $q\le 2/3$.  
Now,  for two operators ${\bf V}$ and ${\bf U}$,   Eq. \eqref{inequseful} entails the bound 
$$ \left| {\rm Tr\, }  {\bf A}{\bf V} - {\rm Tr\, }  {\bf A}{\bf U}\right|   = \left| {\rm Tr\, }  {\bf A}({\bf V} -{\bf U}) \right|  \le \sigma \|{\bf V}- \bf {U}\|\ .$$
   Applying   the previous  relation, with  ${\bf V}_K := \sum_{k = 1}^ K {\bf T}^k$, ${\bf U} := (I-{\bf T})^{-1}$,   together with the bound
$||{\bf V}_K- \mathbf {U}||<\!\!<|| {\bf T}||^K$,   ends the proof.   
\end{proof}

\subsection{ Nuclearity and trace properties of the  transfer operator ${\bf H}_{t, w}$ acting on $\A$.}\label{Sec:Trace}

We now return to the present framework, where the Banach space $\A$,   defined in \eqref{calA},  admits   the Taylor basis $\{(z-1)^m,\, m\in \mathbb Z_{\ge 0}\}$ as a Schauder basis.

 The next result  summarizes the main properties of the various operators of interest, relative to their nuclearity and their traces.  

  \begin{proposition} \label{proptraces}
 $(a)$  For any $h \in {\cal H}^+$,   the weighted transfer operator $\mathbf{H}_{[h],t,w }$ is nuclear  of order  0. Its trace satisfies
 \begin{equation} \label{Eq:Trace-compo}
{\rm Tr\, }\mathbf{H}_{[h],t,w}= {\alpha(h)^{2t}}\frac{  \exp(wc(h))} {1-(-1)^{|h|} \alpha(h)^2}\, .
\end{equation}
(Recall that $|h|$ is the depth of $h$.) 
 
 $(b)$  Consider a pair $(s, w)$ in  the domain $\Gamma(a)$ defined   in \eqref{Eq:Q} with $a>3/4$.   Then, the  complete  operator ${\bf H}_{t, w}$ and its iterates   ${\bf H}^k_{t, w}$  (for $k \ge 1$) are  nuclear operators of order  $q \le  2/3$. One has, for $(s, w) \in \Gamma(a)$, 
\begin{equation} \label{tr3/4}
  {\rm Tr\, } {\bf H}^k_{t, w} = \sum_{ h \in {\cal H}^k} {\rm Tr \, } {\bf H}_{[h], t, w} = \sum_{h \in {\cal H}^k}  {\alpha(h)^{2t}}\frac{  \exp(wc(h))} {1-(-1)^k \alpha(h)^2}, \qquad \hbox{for $k \ge 1$} \, .  
  \end{equation}
 For any $k \ge 1$,  the mapping $(t, w) \mapsto  {\rm Tr\, } {\bf H}^k_{t, w} $ is analytic on $\Gamma(a)$.

 $(c)$ Consider a pair $(t, w)$ in  the domain $\Gamma(a)$  with $a>1$. Suppose, in addition, that  the spectral radius of ${\bf H}_{t,w}$ is less than 1. Then, 
 the quasi-inverse ${\bf H}_{t,w}(I-{\bf H}_{t,w})^{-1}$ and the even quasi-inverse ${\bf E}_{t,w} := {\bf H}^2_{t,w}(I-{\bf H}^2_{t,w})^{-1}$ are nuclear of order $q\le 2/3$. For these pairs $(t,w)$, the following holds,  
 \begin{equation} \label{trqi}
  {\rm Tr\, } {\bf H}_{t,w}(I-{\bf H}_{t,w})^{-1} = \sum_{ k\ge 1}{\rm Tr\, } {\bf H}^k_{t, w} 
  \quad \hbox{ and } \quad 
    {\rm Tr\, } {\bf E}_{t,w} = \sum_{\substack{k {\rm \, even}\\ { k\ge 2}}}{\rm Tr\, } {\bf H}^k_{t, w}\, .
  \end{equation}
Moreover, the mappings $(t,w)\mapsto {\rm Tr\, } {\bf H}_{t,w}(I-{\bf H}_{t,w})^{-1} $ and $(t,w)\mapsto {\rm Tr\, } {\bf E}_{t,w}$ are  analytic on $\Gamma(a)$ with $a>1$, and the following equalities hold 
\begin{align} \label{trqi2}
  {\rm Tr\, } {\bf H}_{t,w}(I-{\bf H}_{t,w})^{-1} &= \sum_{h\in \mathcal{H}^+}  {\alpha(h)^{2t}}\frac{  \exp(wc(h))} {1-(-1)^{|h|} \alpha(h)^2},\\ \nonumber
  {\rm Tr\, } {\bf E}_{t,w} &= \sum_{h\in \mathcal{H}^+,\ |h| {\rm \, even} }  {\alpha(h)^{2t}}\frac{  \exp(wc(h))} {1- \alpha(h)^2}.
  \end{align}

 \end{proposition}
 
 \begin{proof}

$(a)$  Mayer and Momeni-Venkov   have studied the nuclearity of the {\em unweighted} transfer operator.  We follow the approach  described    in Lemma 2.7 of  \cite{MV}, where  the authors    use the Schauder basis of $\A$, formed with the polynomials $(z-1)^m$ with $m \ge 0$. We adapt their method to the case of  the {\em weighted} operator
${\bf H}_{[h], t, w}$ for any $h \in {\cal H}^+$. 

\smallskip   Using the Taylor expansion of $f$  at $z = 1$ in the expression of  ${\bf H}_{[h], t, w} [f](z)$  provides a new expression
 $${\bf H}_{[h], t, w} [f](z) =e^{w c(h)}  \left[\sum_{m = 0} ^\infty  \frac {f^{(m)} (1)}{m!}  \left(h(z)-1\right)^m\right]  \und h^t(z) \, .$$
 that involves  the analytic extensions $\und h$ of $|h'|$. 
We consider here the case $r = 5/4$ in Property $(P3)$, 
 and the related  radii $r, \tilde r$.  For any $h \in {\cal H}^+$, and $t =  \sigma +i \tau$,    we introduce with $(P3)(ii)$,  the bound
$$ a_h(t):= \sup_{z \in {\cal V}} |\und h^{t}(z)| \le e^{\pi |\tau|}  |h' (-1/4)|^\sigma \, .$$
The  following objects  defined for $m \ge 0$ and $h \in {\cal H}^+$,
 $$f_{m, [h]} (z) =  \frac { \und h^t(z) } {a_h(t)}  \ \left( \frac {h(z)-1} {\tilde r}\right)^m , \quad f^\star _{m, [h]} [f] = r^m   \frac {f^{(m)} (1)}{m!} \, ,$$
 satisfy with $(P3)$  and  Cauchy estimates,   the two bounds $$||f_{m, [h]}||  \le 1, \quad   ||f^\star _{m, [h]}|| \le 1 \, , 
$$
 and the decomposition 
 $${\bf H}_{[h], t, w} [f] = \sum_{m= 0}^\infty  \lambda_{m, [h]}\,   f_{m, [h]}\,   f^\star _{m, [h]} [f], \  \hbox { with }\   \lambda_{m, [h]} := \left( \frac {r}{\tilde r}\right) ^{-m}  e^{w c(h)}  a_h(t)  $$
  holds.  As the sequence  $m \mapsto \lambda_{m, [h]}$ is geometric and thus  belongs to $\ell^p$
   for any $0<p \le 1$, this proves   the nuclearity of order 0. 
\smallskip    
   Each component operator $\mathbf{H}_{[h],t,w}$ is  a composition operator of the form $f \mapsto g \cdot f \circ h$  where $g := \exp (w c(h)) \cdot \und  h^{t}$. With Properties $(P3)$, it acts  on $\A$.   Moreover,  due to $(P3)(i)$,   any  branch $h\in {\cal H}^+$  maps
the domain $\mathcal{V}$ strictly inside itself. According to \cite[Lemma 7.10]{Magreen},  the spectrum of such a composition operator acting on $\mathcal{A}_\infty(\mathcal{V})$  is the set 
$$\{  \mu_n : =  g(x_h) \cdot  (h'(x_h))^n,\ \mid  n\in \mathbb{Z}_{\ge 0}\} \, , $$ 
that  involves  the unique fixed point   $x_h$ of $h$. Every eigenvalue is simple. \\
This result applies to  $\mathbf{H}_{[h],t,w}$ and provides  the  expression  given in \eqref{Eq:Trace-compo}  for the trace of ${\bf H}_{[h], t, w}$ 
  which involves  $\alpha(h)$ via \eqref{alpha}.

     $(b)$ 
       For any fixed $h \in {\cal H}^+$,  any $(t, w)$,   and any real $p>0$,   the sequence 
  $ m \mapsto \lambda_{m, [h]}$  satisfies, with Lemma   1, 
  $$\sum_{m \ge 0} |\lambda_{m, [h]}|^p \le    \left( \frac {\tilde r}{r}\right)^{pm}  |e^{\nu p c(h)}| \,    a_h(t)^p <\!\!<   \left( \frac {\tilde r}{r}\right)^{pm} {e^{\pi p |\tau| }}|h'(-1/4)|^{(\sigma-d|\nu|) p} \, .$$
  Consider now  a pair $(t, w)\in \Gamma (a)$ and an exponent $p$ for which  $a\, p:= (\sigma-d|\nu|) p > 1/2$. Then,  
  using  the bound \eqref{Hnbound1/2}  we  obtain, for any $(t, w) \in \Gamma(a)$
  $$ \sum_{(m, h) \in {\mathbb N} \times{\cal H^k}} |\lambda_{m, [h]}|^{p} \le \frac {e^{p \pi |\tau|}} {1- (\tilde r/r) ^{p} } {\bf H}^k_{ap}[1](-1/4)\, .$$
   
 Then,  for  $(t, w)\in \Gamma (a)$,  the operator ${\bf H}^k_{t, w}$ is nuclear of order $q  < 1/(2a)$ and is nuclear of order $q < 2/3$  as soon as 
  $(t,w)\in \Gamma(a)$ with $a> 3/4$. 
  Then,  for $a> 3/4$, we apply Lemma \ref{thm-dream} with $q_0 = 2/3$, and  get  
  \begin{equation}\label{Eq:sumTr}
{\rm Tr\, }{\bf H}_{t,w}^k=\sum_{h\in \mathcal{H}^k} {\rm Tr\, }{\bf H}_{[h], t,w}\, .
\end{equation}
The explicit expression  given in \eqref{Eq:Trace-compo} shows that the mapping $(t,w)\mapsto {\rm Tr\, }\mathbf{H}^k_{t,w}$ is analytic on $\Gamma(a)$.    

$(c)$  Eq. \eqref{trqi} is a direct application of Lemma \ref{useful} $(ii)$ with ${\bf A}=  {\bf T} =  {\bf H}_{t,w}$ or ${\bf A}= {\bf T} =  {\bf H}^2_{t,w}$. When $(t,w) \in \Gamma(a)$ with $a>1$ the spectral radii of ${\bf H}_{t,w}$ and of ${\bf H}^2_{t,w}$   are less than 1 (see Proposition \ref{AA}), hence Relation  \eqref{trqi} holds. On this domain $\Gamma(a)$,  the expression of   ${\rm Tr\, }{\bf H}_{t,w}^k$  given in \eqref{tr3/4} is the general term of a series which is absolutely convergent. This  thus entails 
the relations given  in \eqref{trqi2}.  
 
\end{proof}

\subsection{Spectral dominant properties of the  transfer operator ${\bf H}_{t, w}$ acting on $\A$.}\label{Sec:SpecOp}
For real  pairs $(t, w) \in \Gamma(a)$ (with $a>1/2$), the operator ${\bf H}_{t, w} $ satisfies  strong positive properties that  entail the existence of  dominant spectral objects, in the same vein as the Perron-Frobenius  properties. Also,  for $(t, w) \in \Gamma(a)$, (with $a>1/2$),  the map
 $(t,w)\mapsto 
 {\bf H}_{t, w}$
 is analytic in the sense of Kato \cite[Chapter 3]{Kato}. 
 As ${\bf H}_{t, w}$ is compact,  the dominant spectral objects defined for real pairs may be  extended (with analytic perturbation of the dominant part of the spectrum)  when the pair $(t, w)$ is close to a  real pair.

We are mainly interested in a neighborhood of the real pair $(1, 0)$, and  
 we begin with a complex neighborhood ${\cal T} =   {\cal S} \times {\cal W}$ of $(1, 0)$. For  $(t, w) \in {\cal T}$,    a spectral decomposition for the operator ${\mathbf H}_{t,w}$  holds and   there exist operators $\mathbf{P}_{t,w}$ and $\mathbf{N}_{t,w}$  for which  the operator ${\bf H}_{t, w}$ decomposes as
\begin{equation}\label{Eq:Sp}
 \mathbf{H}_{t,w}=\lambda(t,w)\mathbf{P}_{t,w}+\mathbf{N}_{t,w}
\end{equation}
with $\mathbf{P}_{t,w}\circ \mathbf{N}_{t,w}=\mathbf{N}_{t,w}\circ\mathbf{P}_{t,w}=0$.  As ${\bf H}_{t, w}$ is compact, and 
due to the  equality $\lambda(1, 0) = 1$, the spectral radius $R(t, w)$  of $\mathbf{N}_{t,w}$   is at most equal to $R<1$. 
 The operator $\mathbf{P}_{t,w}$ is rank one,  of the form  \begin{equation} \label {P}
 {\bf P}_{t, w} [f](x) =  f_{t, w}(x) \,  \mu_{t,w}[f] 
 \end{equation}  and involves the dominant eigenfunction $f_{t, w}$ of ${\bf H}_{t, w}$ and the eigenmeasure $\mu_{t, w}$ of the adjoint operator ${\mathbf H}^\star_{t,w}$. Moreover,  with analytic perturbation of  the dominant part of the spectrum,  all the spectral objects depend analytically on ${\cal T}$.  
 
 \smallskip
  The spectral decomposition   extends to the iterates of ${\bf H}^k_{t, w}$, 
\begin{equation}\label{Eq:SpHk} 
 \mathbf{H}_{t,w}^k=\lambda(t,w)^k \, \mathbf{P}_{t,w}+\mathbf{N}^k_{t,w} , \qquad \hbox{for any   $k\ge 1$} \, .
\end{equation}
As the  ``remainder quasi-inverse''    
${\bf N}_{t, w}(I-\mathbf{N}_{t,w})^ {-1}$ is analytic on ${\cal T}$, and   the quotient $ {\lambda(t, w)} / (1-\lambda(t, w))$ is meromorphic on ${\cal T}$, 
 the decomposition of the  quasi-inverse,  
\begin{equation}\label{Eq:qinv}(I-\mathbf{H}_{t,w})^ {-1}=\frac{\lambda(t, w)}{1-\lambda(t,w)} \mathbf{P}_{t,w}+(I-\mathbf{N}_{t,w})^ {-1} \, , \end{equation}
 proves that the quasi-inverse
$(I-\mathbf{H}_{t,w})^ {-1}$ is meromorphic on ${\cal T}$. There is an analog proof for the  ``even quasi-inverse''  ${\bf E}_{t, w}$  defined in \eqref{Eq:EO}, 
\begin{equation} \label{decE}
{\bf E}_{t, w} :=  {\bf H}_{t, w}^2 (I-{\bf H}^2_{t,w})^ {-1}=  \frac{\lambda^2(t, w)}{1-\lambda^2(t,w)} \mathbf{P}_{t,w}+ {\bf N}^2_{t, w}(I-\mathbf{N}^2_{t,w})^ {-1} \, .
\end{equation}

 \subsection {Spectral decomposition and traces.} \label{Sec:SpecTr}
The projector ${\bf P}_{t, w}$ is clearly nuclear  of order 0 (as any operator of finite rank), with a trace equal to 1.  Then,  as ${\cal T}$ is a subset of  $\Gamma(a)$  for some $a >3/4$, and due to  Proposition \ref{proptraces},  all the operators involved are nuclear of order $q < 2/3$.  Moreover, as the dominant eigenvalue $\lambda(t,w)$ and  the traces
${\rm Tr \, } {\bf H}^k_{t, w}$ are   well-defined and analytic, it is the same  for   the operators ${\bf N}^k_{t, w}$: their traces   are well-defined and analytic  on ${\cal T}$.  As    the norm of the operator ${\bf N}_{t, w}$ is strictly less  than 1 on ${\cal T}$,    Lemma \ref{useful}  applies  and entails   that the sequences   
$$
     \left[\sum_{k = 1}^K  {\rm Tr\, } {\bf H} ^{2k}_{t, w}\right] - \sum_{k = 1}^K  \lambda^{2k}(t, w)\, , $$ 
has  a  limit (for $K \to \infty$) that is analytic on ${\cal T}$.  As
$$  \lim_{k \to \infty}  \sum_{k = 1}^K  \lambda^{2k}(t, w) = \frac {\lambda^2(t, w)}  {1-\lambda^2(t, w)}$$ 
is meromorphic on ${\cal T}$, we have shown 
  the  equality 
$$  
 {\rm Tr \, } {\bf H}^2_{t, w}  (I- {\bf H}^2_{t, w})^{-1} = \sum_{k = 1}^\infty    {\rm Tr \, } {\bf H}^{2k}_{t, w}$$   that  defines  a meromorphic  function on ${\cal T}$. In particular, 
 \begin{equation} \label{dectrE}
{\rm Tr \, } {\bf E}_{t, w}= 
 \frac{\lambda^2(t, w)}{1-\lambda^2(t,w)} +  {\rm Tr \, }{\bf N}^2_{t, w}(I-\mathbf{N}^2_{t,w})^ {-1} \, .
\end{equation}
defines a meromorphic function  on ${\cal T}$, whose poles are  located on the curve $\{ (t, w) \in {\cal T} \mid \lambda^2(t, w) = 1\}$ that contains the point (1, 0) and thus the curve  $\{ (t, w) \in {\cal T} \mid \lambda(t, w) = 1\}$.

\subsection{Properties near  $(t, w) = (1, 0)$. } \label{Sec:SpecCo} The operator ${\bf H}_{1, 0}$ is a density transformer, and thus $\lambda(1, 0) = 1$. Furthermore, the dominant eigenfunction  $ f_{1, 0}$ coincides with the Gauss density  $\psi(x):=  (1/\log 2) \cdot 1/(1+x)$. 
 There are simple formulae for the two partial derivatives  of $\lambda(t,w)$ at $(1,0)$. They  are  easy consequences of the spectral equality  $\mathbf{H}_{t,w}[f_{t,w}]=\lambda(t,w)f_{t,w}$,  the fact that  ${\bf H}_{1, 0}$ is a density transformer, and the explicit forms of the derivatives of the operator. One obtains
  \begin{equation} 
  \label {calE} -\lambda_t'(1, 0) =   \int_{{\cal I}} \log | T'(x)|\,   \psi(x) \, dx= \frac{\pi^2}{6\log 2} \, , 
  \end{equation}
 \begin{equation} \label{Ec} 
   \lambda'_w(1,0) =  \sum_{h\in \mathcal{H}} c(h) \int_{h(\cal I)} \psi(x)\, dx  \, .
   \end{equation}
 The constant $- \lambda'_t(1,0) >0$ is  the Kolmogorov entropy (denoted here by ${\cal E}$) and $\lambda'_w(1,0)$ is the average  $\E[c]$ of the cost $c$ with respect to the Gauss density  $\psi$. They are both non zero.

\medskip Since  both derivatives are non zero at $(t, w) = (1, 0)$, there is a  complex neighborhood   ${\cal S}_1 $ of $t = 1$,  a neighborhood  $ {\cal W}_1$ of 0 
 and   a unique analytic 
function   $\sigma: {\cal W}_1 \rightarrow  {\mathbb C}$ 
for which 
$$ \{ (t, w) \in {\cal S}_1 \times  {\cal W}_1  \mid \lambda(t, w) = 1\} = \{ w \in  {\cal W}_1 \mid   \lambda(\sigma(w), w) = 1,\, \sigma(0)=1\}\, , $$  
\begin{equation}\label{Eq:sw} \qquad  \sigma'(w)=-\frac{\lambda'_w(\sigma(w),w)}{\lambda'_t(\sigma(w),w)}\ .\end{equation}  
First,  we  may choose  the neighborhood ${\cal S}_1$  to be a rectangle of the form ${\mathcal{S}_1}=  \{s=\sigma+ i\tau \mid  |\tau|<\tau_1 ,\, |\sigma-1| \le \delta_1  \} $. Second,  ``perturbating''  the inequality $\lambda'_t(1, 0)\not = 0$, and 
taking a  (possibly) smaller  neighborhood ${\cal T}_1 :={\cal S}_1 \times {\cal W}_1$,  exhibits a constant $A>0$ for which 
the ratio $|\lambda(t, w)-1|/|s- \sigma(w)| $  is at least $A>0$ on $\overline{\cal T}_1$ (and  its inverse is then bounded).  
  
Then, for each $w\in {\cal W}_1$, the map $t\mapsto  {\rm Tr\, } {\bf E}_{t, w}$ has  a   unique (simple)  pole on ${\cal S}_1$, located at $s=\sigma(w)$,   with a residue equal to
\begin{equation}
\label{Eq:ResE}\mbox{Res}[ {\rm Tr\, } {\bf E}_{t, w}; s = \sigma(w)]= \frac 1 2  \,\left ( \frac{-1}{\lambda'_t(\sigma(w),w)} \right)\, . \end{equation}

 \subsection*{Acknowledgements} We would like to thank Corinna Ulcigrai for her suggestion to  point out here 
the relation between continued fractions  and  the Farey tessellation of the hyperbolic plane 
 that we have described in  Section \ref{S1.2}. We  are also grateful to  Val\'erie Berth\'e,  Cristian Conde and Pablo Rotondo  for interesting discussions
about rqi numbers and traces of operators. 

{This work was partially funded by  the ANR projects Dyna3S (ANR-13-BS02-0033) and CoDys (ANR-18-CE40-0007),  the STIC-AmSud  Project AleaEnAmSud, and the French-Argentinian
Laboratories in Computer Science  INFINIS  (www.infinis.org) then SINFIN  (www.lia-sinfin.org).}
{Eda Cesaratto was partially supported by  the grant PIO
CONICET-UNGS 14420140100027. }

\end{document}